\documentclass[12pt,leqno,fleqn]{amsart} 
\usepackage[utf8]{inputenc}

\usepackage[fleqn]{mathtools}

\usepackage{amstext,amsxtra,amssymb}
\usepackage{hyperref}
\usepackage{cite}
\usepackage{amsfonts}
\usepackage{pgfplots}
\usepackage{comment}

\usepackage{mathrsfs}

\newcommand{\bc}{\begin{center}}
\newcommand{\ec}{\end{center}}
\newcommand{\beqn}{\begin{align*}}
\newcommand{\eeqn}{\end{align*}}
\newcommand{\benu}{\begin{enumerate}}
\newcommand{\eenu}{\end{enumerate}}
\newcommand{\bit}{\begin{itemize}}
\newcommand{\eit}{\end{itemize}}

\newcommand{\cS }{\mathcal{ S }}

\newcommand{\pdd}[3]{\ifx#2#3\frac{\partial^2 #1}{\partial #2^2}\else \frac{\partial^2 #1}{\partial #2\,\partial #3}\fi}

\newcommand{\ave}[1]{\left\langle #1 \right\rangle}

\newtheorem{theorem}{Theorem}

\newtheorem{lemma}[theorem]{Lemma}
 \newtheorem{prop}[theorem]{Proposition}
 \newtheorem{corollary}[theorem]{Corollary}

\usepackage[top=1.5in, bottom=1.5in, left=1.25 in, right=1.25in]	{geometry}

\usepackage{tikz}
\usetikzlibrary{arrows}

 \author{Robert Kesler}    

\address{1217 21st Street, Santa Monica CA 90404, USA}
\email {robertmkesler@gmail.com}
\begin{document}

 \title[New Improving Properties for Discrete Spherical Maximal Means]{$\ell^p(\mathbb{Z}^d)$-Improving Properties and Sparse Bounds for Discrete Spherical Maximal Means, Revisited}
 
\begin{abstract} We prove an expanded range of $\ell ^{p}(\mathbb{Z}^d)$-improving properties and sparse bounds for discrete spherical maximal means in every dimension $d\geq 6$. Essential elements of the proofs are bounds for high exponent averages of Ramanujan and restricted Kloosterman sums. 
 \end{abstract}
  \maketitle
  \tableofcontents
 \section{Introduction}
 The purpose of this paper is to expand the range of $\ell^p$-improving estimates and sparse bounds for discrete spherical maximal means in every dimension $d \geq 6$ beyond those shown in earlier work of the author \cite{Kesler5}. The new method of proof has been streamlined by invoking the continuous improving $L^p$-estimates of spherical maximal means by a direct transference argument. Before stating our main results, we introduce some notation and background. 
 
   Let $\textbf{A}^d_\lambda$ denote the continuous spherical averaging operator on $\mathbb{R}^d$ at radius $\lambda$, i.e.
\begin{align*}
\textbf{A}^d_\lambda f (x) = \int_{S^{d-1}} f(x-\lambda y) d \sigma (y),
\end{align*}
where $d \geq 2$, $S^{d-1}$ denotes the unit $d-1$ dimensional sphere in $\mathbb{R}^d$ and $\sigma$ is the unit surface measure on $S^{d-1}$.  Stein establishes in \cite{Stein} the spherical maximal theorem for $d \geq 3$, which asserts that $|| \sup_{\lambda} |\textbf{A}^d_\lambda |: L^p(\mathbb{R}^d) \rightarrow L^p(\mathbb{R}^d)|| <\infty$ for all $ \frac{d}{d-1} < p \leq \infty $. The sharp $L^p(\mathbb{R}^d)$-$L^q(\mathbb{R}^d)$ improving result for $\sup_{1 \leq \lambda <2} |\mathbf{A}_\lambda^d|$ is shown by Schlag in \cite{Schlag1}:

\begin{theorem}\label{Thm:0}
Let $d \geq 2$. Define $\mathcal{T}(d)$ to be the interior convex hull of $\{T_{d,j}\}_{j=1}^4$, where  
\begin{align*}
T_{d,1} = (0,1) \qquad &
T_{d,2} = \left(\frac{d-1}{d}, \frac{1}{d} \right) \\ 
T_{d,3} = \left(\frac{d-1}{d}, \frac{d-1}{d} \right) \qquad 
 & T_{d,4} = \left( \frac{ d^2 - d}{d^2+1} ,  \frac{d^2 -d+2}{d^2+1} \right).
\end{align*}
Then for all $(\frac{1}{p}, \frac{1}{r}) \in \mathcal{T}(d)$ 
\begin{align*}
\left| \left| \sup_{1 \leq \lambda <2} | \textbf{A}^d_\lambda |  \right| \right|_{L^p \to L^{r^\prime}}  <\infty.
\end{align*}
By rescaling, for all $(\frac{1}{p}, \frac{1}{r}) \in \mathcal{T}(d)$ there is $A = A(d,p,r)$ and $\Lambda \in 2^{\mathbb{Z}}$ so that
\begin{align*}
\left| \left| \sup_{\Lambda \leq \lambda < 2 \Lambda} |\textbf{A}^d_\lambda  | \right| \right|_{L^p \to L^{r^\prime}}  \leq A \Lambda^{d(1/r^\prime - 1/p)}. 
\end{align*}
\end{theorem}
While the statement
 \begin{align*}
\left| \left| \sup_{1 \leq \lambda <2} | \textbf{A}^d_\lambda | \right| \right|_{L^p \to L^{r^\prime}}  <\infty 
\end{align*}
holds arbitrarily close to $T_{d,1}$ and $T_{d,2}$ along the duality line $\{p = r^\prime\}$ on account of \cite{Stein},  improving properties near $T_{d,3}$ and $T_{d,4}$ require additional argument. In particular, improving properties near $T_{d,4}$ can be obtained by applying the Tomas-Stein restriction theorem to an appropriately constructed Littlewood-Paley decomposition of the spherical means.  
Lacey obtains a sparse extension of the continuous spherical maximal theorem in \cite{Lacey}. To state this result in full rigor, we need to recall some notation for sparse bounds. First, we say a collection of cubes $\cS$ in $\mathbb{R}^d$ is $\rho$-sparse if for each $Q \in \cS$, there is a subset
$E_Q \subset Q$ such that (a) $|E_Q| > \rho |Q|$, and (b) $\|  \sum_{Q \in \cS} 1_{E_Q}  \|_{L^\infty(\mathbb{R}^d)}
\leq \rho^{-1}$.  For a sparse collection $\cS$, a sparse bilinear $(p,r)$-form $\Lambda$ is defined by
\begin{align*}
\Lambda_{\cS,p,r}(f,g) : = \sum_{Q \in \cS} \ave{f}_{Q,p} \ave{g}_{Q,r} |Q|
\end{align*}
where $\ave{h}_{Q,t}:= \left(\frac{1}{|Q|} \sum_{x \in Q} |f(x)|^t \right)^{1/t}$ for any $t:1 \leq t < \infty$, cube $Q \subset \mathbb{Z}^d$, and $h: \mathbb{Z}^d \to \mathbb{C}$. 
Each $\rho$-sparse collection $\cS$ can be split into $O(\rho^{-2})$ many $\frac{1}{2}$-sparse collections; however, as long as $\rho^{-1}=O(1)$, its exact value is not relevant.  For convenience, we also use the following definition introduced in \cite{Culiuc}:  for an operator $T$
acting on bounded and compactly supported functions $f:\mathbb{R}^n \rightarrow \mathbb{C}$ and $1 \leq p, r < \infty$, its sparse norm $\| T: (p,r) \|$ is defined to be the infimum over all $C > 0$ such that for all bounded and compactly supported functions $f,g: \mathbb{R}^n \rightarrow \mathbb{C}$ 
\begin{align*}
|\langle Tf, g \rangle| \leq C \sup_\cS \Lambda_{\cS, p,r} (f,g) 
\end{align*}
where the supremum is taken over all $\frac{1}{2}$-sparse forms. A collection $\mathscr{C}$ of ``cubes" in $\mathbb{Z}^d$ is $\rho$-sparse provided there is a collection $\mathcal{S}$ of $\rho$-sparse cubes in $\mathbb{R}^d$ with the property that $\{R \cap \mathbb{Z}^d : R \in \mathscr{S}\}=\mathscr{C}.$ For a discrete operator $T$, define the sparse norm $||T: (p,r)||$ to be the infimum over all $C>0$ such that for all pairs of bounded and finitely supported functions $f, g : \mathbb{Z}^d \rightarrow \mathbb{C}$ 
\begin{align*}
| \langle T f, g \rangle| \leq C \sup_{\mathcal{S}} \Lambda_{\mathcal{S}, p,r} (f,g)
\end{align*}
where the supremum is taken over all $\frac{1}{2}$-sparse collections $\mathcal{S}$ consisting of discrete ``cubes." 
The sparse bounds obtained for continuous spherical maximal averages by Lacey in \cite{Lacey} are given by

\begin{theorem}\label{Thm:Lacey}
Let $d \geq 2$ and $\mathcal{T}(d)$ be as in Theorem \ref{Thm:0}. Then for all $(\frac{1}{p}, \frac{1}{r}) \in \mathcal{T}(d)$
\begin{align*}
\left| \left| \sup_{\lambda >0} |\textbf{A}^d_\lambda |~ : (p, r) \right| \right| <\infty. 
\end{align*}
\end{theorem}
Magyar, Stein, and Wainger prove a discrete spherical maximal theorem in \cite{Magyar}:  

\begin{theorem}\label{Thm:Magyar}
For each $\lambda \in \tilde{\Lambda}:= \left\{ \lambda >0 : \lambda^2 \in \mathbb{N} \right\}$ define the discrete spherical average 
\begin{align*}
\mathscr{A}_\lambda f(x)= \frac{1}{|\{ |y|=\lambda\}|} \sum_{y \in \mathbb{Z}^d : |y| = \lambda} f(x-y). 
\end{align*}
 Then for all $d \geq 5$ and $ \frac{d}{d-2}< p \leq \infty$
\begin{align*}
\left| \left| \sup_{\lambda \in \tilde{\Lambda}} |\mathscr{A}_\lambda |  \right| \right|_{\ell^p \to \ell^p} <\infty. 
\end{align*}
\end{theorem}

\begin{figure}\label{f:}

\begin{tikzpicture}[scale=4] 
\draw[thick,->] (-.2,0) -- (1.2,0) node[below] {$ \frac 1 p$};
\draw[thick,->] (0,-.2) -- (0,1.2) node[left] {$ \frac 1 r$};
\draw[fill=green] (0,1) --  (.8,.2) 
-- (.8,.8)
--(.75, .85)
-- (0,1) 
;
\draw[fill=teal] (0,1) -- (.8,.9)
--(.8, .8)
--(.75, .85)
-- (0,1)
;
\draw[fill=blue] (0,1) -- (.8,.8)
-- (.75, .85)
-- (0,1)
;
\draw (.8,.05) -- (.8,-.05) node[below] {$ \tfrac{d-2}{d}$};
\draw (.05,.2) -- (-.05,.2) node[left] {$ \tfrac {2}{d}$};
\draw (.05, .8) -- (-.05, .8) node[left] {$ \tfrac {d-2}{d}$};
\draw[loosely dashed] (0,1) -- (1.,1.) node[above] {$ (1,1)$} -- (1.,0); 
\draw (.4,.4) node {$ \mathscr{A}_\lambda $}
;
\end{tikzpicture}

\begin{tikzpicture}[scale=4] 
\draw[thick,->] (-.2,0) -- (1.2,0) node[below] {$ \frac 1 p$};
\draw[thick,->] (0,-.2) -- (0,1.2) node[left] {$ \frac 1 r$};
\draw[fill=orange!80] (.2,.8) --  (.8,.2) 
-- (.8,.8)
--(.2, .8)
;
\draw[fill=purple!80]( .1,.9) -- (.2,.8)
-- (.8, .8)
-- (.75,.85)
--(.1,.9)
;
\draw (.8,.05) -- (.8,-.05) node[below] {$ \tfrac{d-2}{d}$};
\draw (.05,.2) -- (-.05,.2) node[left] {$ \tfrac {2}{d}$};
\draw (.05, .8) -- (-.05, .8) node[left] {$ \tfrac {d-2}{d}$};
\draw (.05,.9) -- (-.25, .9) node[left] {$ \tfrac{ d-1}{d}$};
\draw[loosely dashed] (0,1) -- (1.,1.) node[above] {$ (1,1)$} -- (1.,0); 
\draw (.4,.4) node { $\mathscr{C}_\lambda$ }
;
\end{tikzpicture} 
\begin{tikzpicture}[scale=4] 
\draw[thick,->] (-.2,0) -- (1.2,0) node[below] {$ \frac 1 p$};
\draw[thick,->] (0,-.2) -- (0,1.2) node[left] {$ \frac 1 r$};
\draw[fill=red] ( .12,.88) -- (.2,.8)
-- (.8, .8)
-- (.75,.85)
--(.1,.9)
;
\draw[fill=yellow] (.2,.8) --  (.8,.8) 
-- (.8,.2) 
--(.2,.8)
;
\draw (.05,.88) -- (-.25, .88) node[left] {$ \tfrac{ d-5/2}{d-1}$};
\draw (.12,.05) -- (.12,-.05) node[below] {$ \tfrac {3/2}{d-1}$};
\draw (.8,.05) -- (.8,-.05) node[below] {$ \tfrac{d-2}{d}$};
\draw (.05, .8) -- (-.05, .8) node[left] {$ \tfrac {d-2}{d}$};
\draw[loosely dashed] (0,1) -- (1.,1.) node[above] {$ (1,1)$} -- (1.,0);
\draw (.4,.4) node {$ \mathscr{R}_\lambda $}
;
\end{tikzpicture}

\caption{The blue region in the upper figure is  $\mathcal{R}_*(d) \cap \mathcal{R}(d)^c$ and represents new improving properties for $\sup_{\Lambda \leq \lambda < 2 \Lambda}\left| \mathscr{A}_\lambda \right|$ and sparse bounds for $\sup_{\lambda \in \tilde{\Lambda}} |\mathscr{A}_\lambda|$. The light green region in the upper figure is $\mathcal{R}(d)$ and represent the range of previously known improving properties and sparse bounds for spherical maximal means, while the dark green region in the upper figure represent the range of possible improving properties and sparse bounds for maximal spherical means, for which there are no known positive results or counterexamples.  The purple region in the lower left figure is $\mathcal{Q}_*(d) \cap \mathcal{S}(d)^c$ and represents new improving properties for $\sup_{\Lambda \leq \lambda < 2 \Lambda}\left| \mathscr{C}_\lambda \right|$, where the multiplier $\mathscr{C}_\lambda$ is defined in \ref{Def:c(lambda)}. The orange region is $\mathcal{S}(d)$ and represents previously known improving for $\mathscr{C}_\lambda$. The red region in the lower right figure is $\mathcal{S}_*(d) \cap \mathcal{S}(d)^c$ and represents new improving properties for $\sup_{\Lambda \leq \lambda < 2 \Lambda}\left| \mathscr{R}_\lambda \right|$, where $\mathscr{R}_\lambda=\mathscr{A}_\lambda - \mathscr{C}_\lambda$ is the residual term. The yellow region in the lower right figure is also $\mathcal{S}(d)$ and here represents previously known improving properties and sparse bounds for $\mathscr{R}_\lambda$. } 
\end{figure}
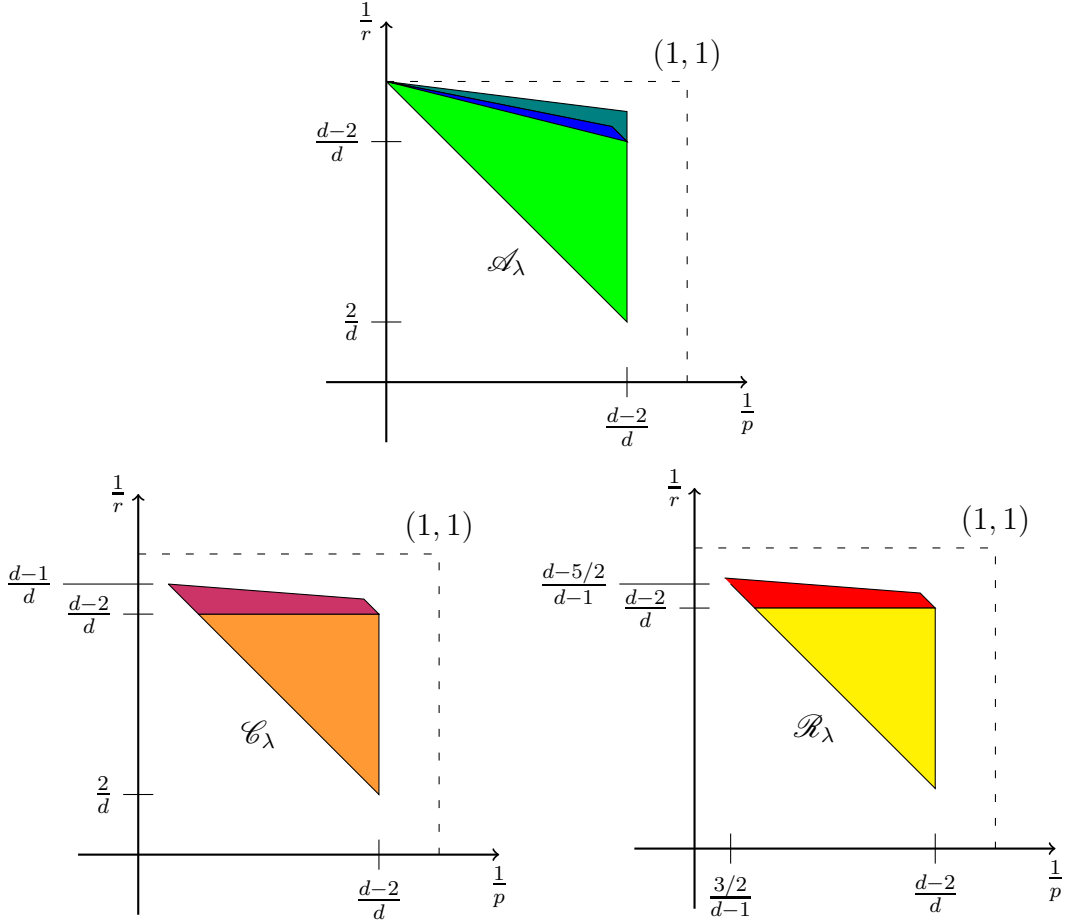
Before stating our main result, we define $\mathcal{Q}_*(d), \mathcal{R}_*(d), \mathcal{S}_*(d), \mathcal{R}(d)$ and $\mathcal{S}(d)$ to be the interior convex hulls of 

\begin{flalign*}
&\mathcal{Q}_{d,1,*}= \left(\frac{1}{d},\frac{d-1}{d} \right) \qquad 
\mathcal{Q}_{d,2,*}=  \left(\frac{d-2}{d}, \frac{2}{d} \right) \qquad 
\mathcal{Q}_{d,3,*} = \left(\frac{d-2}{d}, \frac{d-2}{d} \right) &\\ 
&\mathcal{Q}_{d,4,*} =  \left( \frac{d^2 -d}{d^2+1}\cdot \frac{d-4}{d-2}  +\frac{1}{d-2},  \frac{d^2 -d+2}{d^2+1}\cdot \frac{d-4}{d-2}  +\frac{1}{d-2}  \right),&
\end{flalign*}
\begin{flalign*}
&\mathcal{R}_{d,1,*}= \left(0,1 \right) \qquad
\mathcal{R}_{d,2,*}=  \left(\frac{d-2}{d}, \frac{2}{d} \right) \qquad 
\mathcal{R}_{d,3,*} = \left(\frac{d-2}{d}, \frac{d-2}{d} \right)& \\ 
&\mathcal{R}_{d,4,*} =  \left(   \frac{1}{2} \left(\frac{ d^2 - d}{d^2 + 1} + 1\right )\frac{d-4}{d-1} + \frac{3/2}{d-1}  , \frac{1}{2}\left (\frac{ d^2 - d+2}{d^2 + 1} + 1\right)\frac{d-4}{d-1} + \frac{3/2}{d-1}  \right),&
\end{flalign*}  
\begin{flalign*}
&\mathcal{S}_{d,1,*}= \left(  \frac{3/2}{d-1}  , \frac{d-5/2}{d-1} \right) \qquad 
\mathcal{S}_{d,2,*}=  \left(\frac{d-2}{d}, \frac{2}{d} \right) \qquad 
\mathcal{S}_{d,3,*} = \left(\frac{d-2}{d}, \frac{d-2}{d} \right)& \\ 
&\mathcal{S}_{d,4,*} =  \left(   \frac{1}{2} \left(\frac{ d^2 - d}{d^2 + 1} + 1\right )\frac{d-4}{d-1} + \frac{3/2}{d-1}  , \frac{1}{2}\left (\frac{ d^2 - d+2}{d^2 + 1} + 1\right)\frac{d-4}{d-1} + \frac{3/2}{d-1}  \right),&
\end{flalign*}
\begin{flalign*}
&\mathcal{R}_{d,1}= \left(0,1 \right) \qquad
\mathcal{R}_{d,2}=  \left(\frac{d-2}{d}, \frac{2}{d} \right) \qquad
\mathcal{R}_{d,3} = \left(\frac{d-2}{d}, \frac{d-2}{d} \right), &\\ 
\end{flalign*} 
and  
\begin{flalign*}
&\mathcal{S}_{d,1}= \left(\frac{2}{d},\frac{d-2}{d} \right) \qquad
\mathcal{S}_{d,2} =  \left(\frac{d-2}{d}, \frac{2}{d} \right) \qquad
\mathcal{S}_{d,3} = \left(\frac{d-2}{d}, \frac{d-2}{d} \right),& \\ 
\end{flalign*}
respectively. 
For readers' convenience, these regions are depicted in Figure \ref{f:}. We choose to embellish  $\mathcal{Q}_*(d), \mathcal{R}_*(d)$ and $\mathcal{S}_*(d)$ with a $*$ to differentiate them from the regions $\mathcal{R}(d)$ and $\mathcal{S}(d)$ found in \cite{Kesler5}. 
 Our main theorem strengthens Theorems 4 and 5 in \cite{Kesler5} by extending the improving properties of the spherical means from $\mathcal{R}(d)$ to $\mathcal{R}_*(d)$, the improving properties associated with the ``main term" $\mathscr{C}_\lambda$ from $\mathcal{S}(d)$ to $\mathcal{Q}_*(d)$, and the improving properties associated with the residual term $\mathscr{R}_\lambda$ from $\mathcal{S}(d)$ to $\mathcal{S}_*(d)$.  In doing so, we strengthen the connection between the discrete analogue and the continuous results in \cite{Schlag1} and  \cite{Lacey}. Our main result is the following. 
\begin{theorem}\label{Thm:1}
For all $d \geq 5$ and $(\frac{1}{p}, \frac{1}{r}) \in \mathcal{R}_*(d)$ there exists $A=A(d,p,r)$ such that 
\begin{align}\label{MainEst:Imp}
 \left| \left| \sup_{\Lambda \leq \lambda < 2 \Lambda  } |\mathscr{A}_\lambda |   \right| \right|_{\ell^p \to \ell^{r^\prime}}\leq A \Lambda^{d(1/r^\prime - 1/p)} \qquad \forall \Lambda \in 2^{\mathbb{N}}.
 \end{align}
Moreover, for all $d \geq 5$ and $(\frac{1}p, \frac{1}r ) \in \mathcal{R}_*(d)$
\begin{align}\label{Est:Sp}
\left| \left| \sup_{\lambda \in \tilde{\Lambda}} \left| \mathscr{A}_\lambda \right|: (p,r)  \right| \right| <\infty.
\end{align}
\end{theorem}
The sparse bounds \eqref{Est:Sp} in Theorem \ref{Thm:1} follow from the improving properties \eqref{MainEst:Imp} and the reduction to restricted weak type sparse bounds, as shown for instance by Theorem 16 in \cite{Kesler5}. For completeness, we include the argument for the sparse bound in \S{6}-\S{8}. 
The reader may readily check that $\mathcal{R}_*(d) \supsetneq\mathcal{R}(d)$ for all $ d \geq 6$, while $\mathcal{Q}_*(d) \supsetneq \mathcal{S}(d)$ and $\mathcal{S}_*(d) \supsetneq \mathcal{S}(d)$ for all $d \geq 5$.  

From \cite{Kesler5}, $\max\left\{ \frac{1}{p} + \frac{2}{d}, \frac{1}{r} +\frac{2}{pd} \right\} \leq 1$ is a necessary condition for both \eqref{MainEst:Imp} and  \eqref{Est:Sp} to hold and is of course satisfied by every point in $\mathcal{R}_*(d)$. There still remains a small subset of $[0,1]^2$ where no positive result or counterexample for the improving properties of the discrete spherical maximal means is known. 
 \section*{Acknowledgment}
I thank Michael Lacey for suggesting a ``4 corners" discrete spherical result and providing helpful feedback. 
\section{Decomposing Spherical Means}
We now recall the decomposition of the discrete spherical average $\mathscr{A}_\lambda = \mathscr{C}_\lambda + \mathscr{R}_\lambda$ as first formulated in \cite{Magyar}. The symbol of the multiplier $\mathscr{A}_\lambda$ for $\Lambda \leq \lambda < 2 \Lambda$ and $\Lambda \in 2^{\mathbb{N}}$ can be written for all $\xi \in [-1/2, 1/2)^d$ as
\begin{align}\label{Def:a(lambda)}
a_\lambda (\xi) =& \sum_{q=1}^\Lambda \sum_{a \in \mathbb{Z}^\times_q } a_\lambda^{a/q}(\xi)
\end{align}
where 
\begin{align}
a_\lambda^{a/q}(\xi) =& e^{-2 \pi i \lambda^2 a/q} \sum_{\ell \in \mathbb{Z}^d} G(a/q, \ell) J_\lambda (a/q, \xi - \ell /q) \\
 G(a/q, \ell)=& \frac{1}{q^d} \sum_{n \in \mathbb{Z}^d/q \mathbb{Z}^d} e^{ 2\pi i |n|^2 a/q} e^{-2 \pi i n \cdot l /q}\label{Def:GS} \\ J_\lambda (a/q , \xi) =& \frac{ e^{ 2 \pi}}{\lambda^{d-2}} \int_{I(a,q)} e^{-2 \pi i \lambda^2 \tau} \frac{ e^{ \frac {- \pi |\xi|^2}{2 ( \epsilon - i \tau)} }}{(2 (\epsilon - i \tau))^{d/2} }d \tau  \\ \epsilon =& \frac{1}{\lambda^2} 
\end{align}
and 
$I(a,q)=\left[ - \frac{\beta}{ q \Lambda}, \frac{\alpha}{ q \Lambda} \right]$, $\alpha= \alpha(\frac{a}q, \Lambda) \simeq 1, \beta = \beta(\frac{a}q, \Lambda) \simeq 1$.  
The approximate lengths of the above Farey intervals are enough to prove Theorems 4 and 5 from \cite{Kesler5}, but to show the expanded range of residual term bounds in Theorem \ref{Thm:1} of this paper, we need to discuss their precise lengths, which is accomplished in \S{4}. Another important fact for us is the Gauss sum estimate 
\begin{align}\label{Est:GS}
\left| G(a/q, \ell) \right| \leq A q^{-d/2}
\end{align}
which holds uniformly in $a \in \mathbb{Z}^\times_q, \ell \in \mathbb{Z}^d/q\mathbb{Z}^d$ and $q \in \mathbb{N}$; this is well-known in the $d=1$ case from which the $d \geq 2$ case immediately follows. 
Next, we shall pick $\Phi \in C^\infty ([-1/4, 1/4]^d)$ such that $\Phi\equiv 1$ on $[-1/8, 1/8]^d$ and $\Phi \geq 0$. Then for all $q \in \mathbb{N}$ set $\Phi_q(\xi) = \Phi\left(q \xi \right)$ and define
\begin{align}
b_\lambda (\xi) =& \sum_{q=1}^\Lambda \sum_{a \in \mathbb{Z}^\times_q} b_\lambda^{a/q} (\xi) \label{Def:b(lambda)} \\ 
b_\lambda^{a/q} (\xi) =& e^{-2 \pi i \lambda^2 a/q} \sum_{\ell \in \mathbb{Z}^d/q \mathbb{Z}^d} G(a/q, \ell) \Phi_q(\xi - \ell /q)J_\lambda (a/q, \xi - \ell /q) \nonumber 
\end{align}
along with $\mathscr{B}^{a/q}_\lambda: f \mapsto f * \check{b}_\lambda^{a/q}$ and $\mathscr{B}_\lambda: f \mapsto f * \check{b}_\lambda$. 
So, $b_\lambda^{a/q}$ is constructed from $a_\lambda^{a/q}$
by inserting cutoff factors into each summand of $a_\lambda^{a/q}$ at frequency length scale $\frac{1}{q}$. We subsume the difference $b_\lambda - a_\lambda$ into the residual term $\mathscr{R}_\lambda$. Lastly, we extend the domain of integration in the definition of $J_\lambda$ to all of $\mathbb{R}$ and subsume this difference as part of the residual term $\mathscr{R}_\lambda$. To this end, we introduce
\begin{align*}
I_\lambda (a/q , \xi) = \frac{ e^{ 2 \pi}}{\lambda^{d-2}} \int_{-\infty}^\infty e^{-2 \pi i \lambda^2 \tau} \frac{  e^{ \frac {- \pi |\xi|^2}{2 ( \epsilon - i \tau)} }}{(2 (\epsilon - i \tau))^{d/2}}d \tau 
\end{align*}
and let
\begin{align}
c_\lambda (\xi) =& \sum_{q=1}^\Lambda \sum_{a \in \mathbb{Z}_q^{\times}} c_\lambda^{a/q} (\xi) \label{Def:c(lambda)} \\ 
c_\lambda^{a/q} (\xi) =&  e^{-2 \pi i \lambda^2 a/q} \sum_{\ell \in \mathbb{Z}^d/q\mathbb{Z}^d} G(a/q, \ell) \Phi_q(\xi - \ell /q)I_\lambda ( \xi - \ell /q) 
\end{align}
along with $\mathscr{C}^{a/q}_\lambda: f \mapsto f * \check{c}_\lambda^{a/q}$, and $\mathscr{C}_\lambda : f \mapsto f* \check{c}_\lambda$.
The reason for extending the integral in $J_\lambda$ is that $I_\lambda  = c_d \widehat{d \sigma_\lambda}$, where $c_d$ is a dimensional constant and $d \sigma_\lambda$ is the unit surface measure of the sphere in $\mathbb{R}^d$ of radius $\lambda$. 
This important fact is established in \cite{Magyar}. We thereby observe the identity 
\begin{align*} 
c_\lambda (\xi) = c_d  \sum_{q =1}^\Lambda \sum_{a \in \mathbb{Z}_q^\times} e^{-2 \pi i \lambda^2 a/q} \sum_{\ell \in \mathbb{Z}^d/q \mathbb{Z}^d} G(a/q, \ell) \Phi_q(\xi - \ell /q)\widehat{d\sigma_\lambda} ( \xi - \ell /q). 
\end{align*}

\section{New Improving Properties for $\sup_{\Lambda \leq \lambda < 2 \Lambda} | \mathscr{C}_\lambda|$}
Our goal in this section is to establish improving properties for $\sup_{\lambda \in \tilde{\Lambda}} | \mathscr{C}_\lambda|$ in the expanded range $\mathcal{Q}_*(d)$. To this end, we shall first prove a fact about high exponent averages of Ramanujan sums.

 \begin{lemma}\label{L:Main}
For all $\delta>0$ and $k \in \mathbb{N}$ there is $A = A(\delta, k)$ such that for all $Q \in 2^{\mathbb{N}}$, $M \geq Q^k$, and $N \in \mathbb{Z}$
 \begin{align}\label{Est:MainL}
\left[ \frac{1}{M} \sum_{n=N}^{N+M} \left[ \sum_{Q \leq q < 2 Q} |c_q(n)| \right]^k \right]^{1/k} \leq A Q^{1+\delta}
\end{align}
where $c_q(n):= \sum_{(a,q)=1} e^{2 \pi i \frac{a}{q} n }$  is the Ramanujan sum.
 \end{lemma}
 \begin{proof}
 We may assume without loss of generality that $N \geq 1$. Begin by observing that for fixed $(q_1,\ldots, q_k) \in  [Q , 2Q) ^k$
 \begin{align}\label{Est:0}
 \frac{1}{M} \sum_{n=N}^{N+M} \prod_{j=1}^k  |c_{q_j}(n)|  \leq  \frac{A}{\mathcal{L}(\vec{q})} \sum_{n=1}^{\mathcal{L}(\vec{q})} \prod_{j=1}^k  (q_j,n) ,
 \end{align}
 where $\mathcal{L}(\vec{q})$ is the least common multiple of $(q_1, \ldots, q_k)$. This follows from the condition $M \geq Q^k$ and the bound $|c_q(n)| \leq (q,n)$ valid for all $q,n \in \mathbb{N}$. Applying  \eqref{Est:0} then yields
\begin{align}\label{Est:.5}
\frac{1}{M} \sum_{n=1}^M \left[ \sum_{Q \leq  q  < 2Q} |c_q(n)| \right]^k \leq  \sum_{\vec{q} \in  [Q , 2Q)^k}  \frac{A}{\mathcal{L}(\vec{q})} \sum_{n=1}^{\mathcal{L}(\vec{q})} \prod_{j=1}^k  (q_j,n). 
\end{align}
Let $\epsilon>0$.  Observe that for each $\vec{q} \in [Q, 2Q)^k$ 
\begin{align}\label{Est:100}
 \sum_{n=1}^{\mathcal{L}(\vec{q})} \prod_{j=1}^k  (q_j, n)  \leq A_{\delta,k}  Q^{k(1+\delta)}. 
 \end{align}
By \eqref{Est:100}, the right side of \eqref{Est:.5} is
\begin{align*}
O_{\delta,k} \left( \sum_{\vec{q} \in  [Q , 2Q) ^k} \frac{Q^{k(1+\delta)}}{\mathcal{L}(\vec{q})} \right).
\end{align*}
To show Lemma \ref{L:Main}, it therefore suffices to prove 
\begin{align}\label{Est:2}
\sum_{\vec{q} \in [Q , 2Q) ^k} \frac{1}{\mathcal{L}(\vec{q})} \leq A_{\delta,k } Q^{\delta k}. 
\end{align}
To this end, assume a number $L$ has prime factorization $\prod_{j=1}^M p_j^{n_j} $ with all primes $p_j <2Q$. Then the number of $\vec{q}  \in   [Q , 2Q) ^k$ with $\mathcal{L}(\vec{q}) =L$ is  at most $\left[ \prod_{j=1}^M n_j \right]^k$, which is $O_{\delta,k }(L^{\delta })$. Letting 
\begin{align*}
\mathcal{L}_Q :=\bigcup_{\vec{q} \in  [Q , 2Q) ^k} \{\mathcal{L}(\vec{q})\},
\end{align*}
we note by the previous observation that  \eqref{Est:2} will follow from 
\begin{align}\label{Est:3}
\sum_{L \in \mathcal{L}_Q} \frac{1}{L^{1-\delta}} \leq A_{\delta  } Q^{\delta k}.
\end{align}
However, \eqref{Est:3} follows quickly from the fact that $Card( \mathcal{L}_Q) < Q^k$, and so
\begin{align*}
\sum_{L \in \mathcal{L}_Q} \frac{1}{L^{1-\delta }} \leq \sum_{L =1}^{Card(\mathcal{L}_Q)} \frac{1}{L^{1-\delta}} \leq A_{\delta } Q^{\delta k }.
\end{align*}

\end{proof}
Lemma \ref{L:Main} is an essential element in the proof of the following. 
\begin{theorem}\label{Thm:MainImp}
For all $d \geq 5$ and $\left(\frac{1}{p}, \frac{1}{r}\right) \in \mathcal{Q}_*(d)$ there exists $A=A(d,p,r)$ such that 
\begin{align}\label{Est:Cimp2}
\left \lVert \sup_{\Lambda \leq \lambda  < 2 \Lambda} |  \mathscr{C}_\lambda|\right \rVert_{\ell^p \to \ell^{r^\prime}} \leq A \Lambda^{d \left(\frac{1}{r^\prime} - \frac{1}{p} \right)}  \qquad \forall \Lambda \in 2^{\mathbb{N}}.
\end{align}
\end{theorem}
\begin{proof}
We shall interpolate a favorable $\ell^2 \to \ell^2$ estimates against a ``boundary" estimate for which we pay a satisfactorily small price.  By the transference and factorization argument from \cite{Magyar}, we have for every $1 \leq q \leq \Lambda$ and $a \in \mathbb{Z}^\times_q$
\begin{align*}
\left \lVert \sup_{\Lambda \leq \lambda  < 2 \Lambda} \left |  \mathscr{C}^{a,q}_\lambda \right| \right \rVert_{\ell^2 \to \ell^2} \leq A q^{-d/2}. 
\end{align*}
Therefore, summing on $a \in \mathbb{Z}^\times_q$ and $q \in [Q, 2Q)$ yields for every $Q , \Lambda \in 2^{\mathbb{N}}: Q \leq \Lambda$
\begin{align}\label{Est:C2}
\left \lVert \sup_{\Lambda \leq \lambda  < 2 \Lambda} \left| \sum_{Q \leq q  < 2Q} \sum_{a \in \mathbb{Z}^\times_q} \mathscr{C}^{a,q}_\lambda\right| \right \rVert_{\ell^2 \to \ell^2} \leq A Q^{2-d/2}.
\end{align}
To produce the ``boundary" estimate, we work with the kernel of $\sum_{Q \leq q <2 Q} \sum_{a \in \mathbb{Z}^\times_q} \mathscr{C}^{a,q}_{\lambda}$, which we write as 
\begin{align}\label{Est:CKer}
K^{\mathscr{C}}_{Q,\lambda} (x) :=  \sum_{Q \leq q <2 Q} c_q(|x|^2 - \lambda^2) \cdot d \sigma_\lambda * \check{\Phi}_q(x) \qquad \forall x \in \mathbb{Z}^d. 
\end{align}
Before showing Theorem \ref{Thm:MainImp}, we proceed to show that for every $\delta >0$ there is $A= A(d,\delta)$ such that for all $ \Lambda, Q \in 2^{\mathbb{N}}: Q \leq \Lambda  \leq \lambda < 2 \Lambda, f : \mathbb{Z}^d \to \mathbb{C}$, and $x \in \mathbb{Z}^d$
\begin{align}\label{Est:1}
\left| f *  K^{\mathscr{C}}_{Q,\lambda} (x) \right| \leq AQ^{1+\delta}  \sum_{ l \in \mathbb{N}} \frac{1}{2^{2dl}}  \left( |f|^{1+\delta} * d \sigma_\lambda *(|\check{\Phi}_1| + |\check{\Phi}_{2^l Q}| )(x) \right)^{\frac{1}{ 1+\delta}} .
\end{align}
To parse the meaning of the right side of the above display, the convolution between $d \sigma_\lambda$ and $|\check{\Phi}_1| + |\check{\Phi}_{2^l Q}|$ is understood in the continuous sense and taken before the convolution with $|f|^{1+\delta}$, which is taken in the discrete sense.  We shall begin the proof of \eqref{Est:1} by introducing the family of regions
\begin{align*}
\mathcal{S}_{Q,\lambda}^{-1} :=& \left\{ x \in \mathbb{Z}^d:  \left | \frac{ |x| - \lambda}{Q} \right| < 1 \right\}  \\ 
\mathcal{S}_{Q,\lambda}^l :=& \left\{ x \in \mathbb{Z}^d:  2^l \leq \left | \frac{ |x| - \lambda}{Q} \right| < 2^{l+1} \right\} \qquad \forall l \geq 0.
\end{align*}
By the triangle inequality, 
\begin{align*}
\left| f *  K^{\mathscr{C}}_{Q,\lambda} (x) \right| \leq \left| f * (1_{\mathcal{S}^{-1}_{Q, \lambda}} K^{\mathscr{C}}_{Q,\lambda}) (x) \right|  +  \sum_{l = 0}^\infty \left| f * (1_{\mathcal{S}^l_{Q, \lambda}} K^{\mathscr{C}}_{Q,\lambda}) (x) \right| . 
\end{align*}
The kernel $K^{\mathscr{C}}_{Q, \lambda}$ is largest on the set $\mathcal{S}^{-1}_{Q,\lambda}$, which contribution we handle first.  
We now split into two cases depending the relative sizes of $Q$ and $\lambda$. We fix $k \in \mathbb{N}$ to be determined later for the purposes of applying Lemma \ref{L:Main}. By the triangle inequality, 
\begin{align}\label{Est:20}
f * (1_{\mathcal{S}^{-1}_{Q, \lambda}} K^{\mathscr{C}}_{Q,\lambda})(x) \leq \frac{A}{ Q \Lambda^{d-1}} \sum_{ \substack{ n \in \mathbb{N}\\ |\sqrt{n}   - \lambda| < Q} } \sum_{ |y|^2 = n} \sum_{Q \leq q <2 Q} |f (x-y) c_q(n - \lambda^2)|.
\end{align}
We now assume, in addition, that $\Lambda > Q^k$. The case when $\Lambda \leq Q^k$ is even shorter and will be handled separately. We now use H\"{o}lder's inequality for the sum on $y$ with fixed radius to majorize the right side of the above display by
\begin{align*}
\frac{A}{ Q \Lambda} \sum_{ \substack{ n \in \mathbb{N}\\ |\sqrt{n}   - \lambda| < Q} } \left[\frac{1}{\Lambda^{d-2}}\sum_{ |y|^2 = n} |f (x-y)|^{\frac{k}{k-1}} \right]^{\frac{k-1}k} \sum_{Q \leq q <2Q}  |c_q(n - \lambda^2)|.
\end{align*}
Applying H\"{o}lder's inequality to the sum on $n$ bounds the above display by
\begin{align*}
\frac{A}{ Q \Lambda} \left[ \sum_{y \in \mathcal{S}_{Q, \lambda}^{-1}}   \frac{1}{\Lambda^{d-2}} |f (x-y)|^{\frac{k}{k-1}} \right]^{\frac{k-1}k} \left[ \sum_{ \substack{ n \in \mathbb{N}\\ |\sqrt{n}   - \lambda| < Q} } \left[ \sum_{Q \leq q<2 Q} |c_q(n - \lambda^2)| \right]^k \right]^{1/k}.
\end{align*}
Because $\Lambda > Q^k$, Lemma \ref{L:Main} ensures that the right most factor in the above display is $O_{k,\delta}( (Q \Lambda)^{1/k} Q^{1+\delta})$. An immediate consequence of this fact is 
\begin{align*}
\left| f *  (1_{\mathcal{S}^{-1}_{Q, \lambda}} K^{\mathscr{C}}_{Q,\lambda}) (x) \right| \leq A \left[ \frac{1}{Q \Lambda^{d-1}} \sum_{y \in \mathcal{S}_{Q, \Lambda}^{-1}}   |f (x-y)|^{\frac{k}{k-1}} \right]^{\frac{k-1}k},
\end{align*}
which is an acceptable contribution to \eqref{Est:1} provided $k = k(\delta)$ is taken sufficiently large. If $ \Lambda \leq Q^k$, we use the quickly verified fact that for every $\delta >0$  and $|x| \not = \lambda$
\begin{align}\label{Est:Ram1}
\left| \sum_{Q \leq q <2 Q} c_q(|x|^2 - \lambda^2) \right| \leq A_{\epsilon} Q  \Lambda^{\delta} \leq A_{\epsilon} Q^{1+\delta k}.  
\end{align}
Therefore, when $\Lambda \leq Q^k$, we may use \eqref{Est:20}to obtain that for every $\delta>0,x \in \mathbb{Z}^d$  
\begin{align*}
\left| f *  (1_{\mathcal{S}^{-1}_{Q, \lambda}}1_{|\cdot| \not = \lambda} K^{\mathscr{C}}_{Q,\lambda}) (x) \right|  \leq A_{k, \epsilon} Q^{1+\epsilon}  |f|  * d \sigma_\lambda * |\check{\Phi}_Q|(x).
\end{align*}
If $|x| = \lambda$, the estimate  $\left| \sum_{Q \leq q <2 Q} c_q(|x|^2 - \lambda^2) \right|  \leq Q^2$ and \eqref{Est:20} yield for all $x \in \mathbb{Z}^d$
\begin{align*}
\left| f *  (1_{\mathcal{S}^{-1}_{Q, \lambda}}1_{|\cdot|  = \lambda} K^{\mathscr{C}}_{Q,\lambda}) (x) \right|  \leq A Q |f|  * d \sigma_\lambda * |\check{\Phi}_1|(x).
\end{align*}
To show \eqref{Est:1}, it therefore suffices to show for all $\delta>0,l \geq 0, x \in \mathbb{Z}^d$
\begin{align*}
\left| f *( 1_{\mathcal{S}^l_{Q, \lambda}}  K^{\mathscr{C}}_{Q,\lambda} )(x) \right| \leq A \frac{Q^{1+\delta}}{2^{dl}} \left( |f|^{1+\delta} * d \sigma_\lambda * |\check{\Phi}_{2^lQ}|(x) \right)^{\frac{1}{1+\delta}}.
\end{align*}
By the rapid decay of the $K^{\mathscr{C}}_{Q, \lambda}$ away from $\mathcal{S}^k_{Q,l}$, we may observe 
\begin{align*}
1_{ \mathcal{S}^l_{Q, \lambda}} (x) \left| K^{\mathscr{C}}_{Q, \lambda} (x) \right| \leq \frac{A}{2^{10l}} \frac{1}{Q \Lambda^{d-1}} \qquad \forall x \in \mathbb{Z}^d, l \geq 0. 
\end{align*}
Similar to before, we have by the triangle inequality that $\forall x \in \mathbb{Z}^d$
\begin{align}\label{Est:21}
& \left| f * (1_{\mathcal{S}^{l}_{Q, \lambda}} K^{\mathscr{C}}_{Q,\lambda})(x) \right|  \\ \leq &  \frac{A}{ Q \Lambda^{d-1}  2^{2d l}} \sum_{ \substack{ n \in \mathbb{N}\\ |\sqrt{n}   - \lambda| < Q 2^{l+1}} } \sum_{ |y|^2 = n} \sum_{Q \leq q <2 Q} |f (x-y) c_q(n - \lambda^2)| \nonumber.
\end{align}
To handle the case when $\Lambda > Q^k$, first observe from \eqref{Est:21} and Lemma \ref{L:Main} that
\begin{align*}
\left| f *  (1_{\mathcal{S}^{l}_{Q, \lambda}} K^{\mathscr{C}}_{Q,\lambda}) (x) \right| \leq & A  \frac{Q^{1+\delta}}{2^{2 d l}} \left[ \frac{2^{l/(k-1)}}{ Q \Lambda^{d-1}} \sum_{y \in \mathcal{S}_{Q, \Lambda}^{l}}   |f (x-y)|^{\frac{k}{k-1}} \right]^{\frac{k-1}k} \qquad \forall x \in \mathbb{Z}^d.
\end{align*}
If $2^l \leq \frac{\Lambda}{ Q}$, then for sufficiently large $k = k(\delta)$ and all $x \in \mathbb{Z}^d$
\[  \frac{1}{2^{2dl}}   \left[ \frac{2^{l/(k-1)} }{ 2^l Q \Lambda^{d-1}} \sum_{y \in \mathcal{S}_{Q, \Lambda}^l}   |f (x-y)|^{\frac{k}{k-1}} \right]^{\frac{k-1}k} \leq \frac{A}{2^{dl}} ( |f|^{1+\delta} * d \sigma_\lambda * |\check{\Phi}_{2^l Q}| (x))^{\frac{1}{1+\delta}}.\] 
If $2^l > \frac{\Lambda}{Q}$, then for sufficiently large $k = k(\delta)$ and all $x \in \mathbb{Z}^d$.
\begin{align*}
\frac{1}{ 2^{2dl}}  \left[ \frac{2^{l/(k-1)}} {Q \Lambda^{d-1}} \sum_{y \in \mathcal{S}_{Q, \Lambda}^l}   |f (x-y)|^{\frac{k}{k-1}} \right]^{\frac{k-1}k}   \leq & \frac{1}{ 2^{dl}}  \left[ \frac{1} {(2^lQ)^d} \sum_{y \in \mathcal{S}_{Q, \Lambda}^l}   |f (x-y)|^{\frac{k}{k-1}} \right]^{\frac{k-1}k}  \\ \leq & \frac{A}{2^{dl}} ( |f|^{1+\delta} * d \sigma_\lambda * |\check{\Phi}_{2^l Q}| (x))^{\frac{1}{ 1+\delta}}.
\end{align*}
For $\Lambda \leq Q^k$, we may use \eqref{Est:Ram1} and \eqref{Est:21} to note that for every $\delta >0$
\begin{align*}
\left| f *  (1_{\mathcal{S}^{l}_{Q, \lambda}}  K^{\mathscr{C}}_{Q,\lambda}) (x) \right|  \leq \frac{A_{k, \delta}}{2^{dl}} Q^{1+\delta}  |f|  * d \sigma_\lambda * |\check{\Phi}_{2^l Q}| (x) \qquad \forall x \in \mathbb{Z}^d.
\end{align*}
By taking $k = k (\delta) $ sufficiently large, we obtain estimate \eqref{Est:1}. 
It is worth noting that if we had instead used the trivial bound $|c_q(n)| <q$ to control \eqref{Est:CKer}, then we arrive at \begin{align*}
\left| f * K^{\mathscr{C}}_{Q,\lambda} (x)  \right| \leq  AQ^2 |f|* d \sigma_\lambda * |\check{\Phi}_Q | (x) \qquad \forall x \in \mathbb{Z}^d,
\end{align*}
which fails to produce any results beyond those in \cite{Kesler5}. 
We now use estimate \eqref{Est:1} to prove Theorem \ref{Thm:MainImp}. By comparing the discrete operator on the right side of \eqref{Est:1} to its continuous analogue and invoking Theorem \ref{Thm:1} from \cite{Schlag1}, it follows that for all $\left(\frac{1}{p}, \frac{1}{r} \right) \in \mathcal{T}(d)$ and $\delta >0$, there is $A=A(d,p,r, \delta)$ such that
\begin{align}\label{Est:101}
\left \lVert \sup_{\Lambda \leq \lambda  < 2 \Lambda} \left| \sum_{Q \leq q  <2 Q} \sum_{a \in \mathbb{Z}^\times_q} \mathscr{C}^{a,q}_\lambda\right|  \right \rVert_{\ell^p \to \ell^{r^\prime}} \leq A Q^{1+\delta} \qquad  \forall Q \leq \Lambda. 
\end{align}
Next, by interpolating \eqref{Est:101} for $(\frac{1}{p}, \frac{1}{q})$ arbitrarily close to $\partial \mathcal{T}(d)$ with \eqref{Est:C2}, we obtain that for all $\left(\frac{1}{p}, \frac{1}{r} \right) \in \mathcal{T}(d)$, there are $A=A(d,p,r)$ and $\eta = \eta(d,p,r)>0$ such that
\begin{align}\label{Est:4}
\left \lVert \sup_{\Lambda \leq \lambda  < 2 \Lambda} \left| \sum_{Q \leq q  <2 Q} \sum_{a \in \mathbb{Z}^\times_q} \mathscr{C}^{a,q}_\lambda\right| \right \rVert_{\ell^p \to \ell^{r^\prime}} \leq A Q^{-\eta} \qquad  \forall Q \leq \Lambda. 
\end{align}
Summing \eqref{Est:4} on $Q \in 2^{\mathbb{N}}: Q  \leq \Lambda$ yields estimate \eqref{Est:Cimp2} for all $(\frac{1}p, \frac{1}{r}) \in \mathcal{Q}_*(d)$ .

\end{proof}

 \section{New Improving Properties for $\sup_{\Lambda \leq \lambda < 2 \Lambda} | \mathscr{R}_\lambda|$}
 
Our main goal in the section is to prove the following:
\begin{theorem}\label{Thm:ResImp}
For all $d \geq 5$ and $(\frac{1}p, \frac{1}r ) \in \mathcal{S}_*(d)$
\begin{align}
\left \lVert \sup_{\lambda \in \tilde{\Lambda}} |  \mathscr{R}_\lambda |\right \rVert_{\ell^p \to \ell^{r^\prime}} \leq A \Lambda^{d (\frac{1}{r^\prime} - \frac{1}p)} \qquad \forall \Lambda \in 2^{\mathbb{N}}.
\end{align}
\end{theorem}
Recalling $\mathscr{A}_\lambda, \mathscr{B}_\lambda,$ and $\mathscr{C}_\lambda$ as given in \eqref{Def:a(lambda)}, \eqref{Def:b(lambda)}, and \eqref{Def:c(lambda)},
we proceed to decompose $\mathscr{R}_\lambda = (\mathscr{A}_\lambda - \mathscr{B}_\lambda) + (\mathscr{B}_\lambda - \mathscr{C}_\lambda)$
and study $\sup_\lambda | \mathscr{A}_\lambda - \mathscr{B}_\lambda|$ and $\sup_{\lambda} |\mathscr{B}_\lambda - \mathscr{C}_\lambda|$ along the lines of \cite{Kesler5}, except that we focus a bit on the structure of the Farey intervals arising in \eqref{Def:c(lambda)} and exploit known bounds for Kloosterman sums by recording the following proposition. 
\begin{prop}
For all $\Lambda \in 2^{\mathbb{N}}, 1 \leq q \leq \Lambda,$ and $ \tau \in \mathbb{T}$,  $\exists N_1(q, \Lambda, \tau), N_2(q, \Lambda, \tau) \in \mathbb{Z}^\times_q $ such that
\begin{align*}
 \{ a \in \mathbb{Z}^\times_q : I(a,q) \ni \tau\} = \{ a \in \mathbb{Z}^\times_q:  N_1( q, \Lambda, \tau) \leq a_q^{-1}  \leq N_2(q, \Lambda, \tau)\}. 
\end{align*}
\end{prop}
\begin{proof}
Begin by observing that if $\frac{a}{q}$ and $\frac{\tilde{a}}{\tilde{q}}$ are Farey neighbors at a given level $\Lambda$, by which we mean there is no Farey point $\frac{\bar{a}}{\bar{q}}$ satisfying 
\begin{align*}
\frac{a}{q} < \frac{\bar{a}}{\bar{q}} < \frac{\tilde{a}}{\tilde{q}},
\end{align*}
then $| \frac{a}q - \frac{\tilde{a}} {\tilde{q}} | = \frac{1}{q \tilde{q}}$. Therefore, we are free to choose the Farey intervals so that $I(a,q)$ extends in the direction towards $\frac{\tilde{a}}{\tilde{q}}$ a distance $\frac{1}{q} \frac{1}{q + \tilde{q}}$ and $I(\tilde{a}, \tilde{q})$ extends towards $\frac{a}{q}$ a distance $\frac{1}{\tilde{q}} \frac{1}{q + \tilde{q}}$. Moreover, for Farey point $\frac{a}{q}$, the left or right Farey neighbor $\frac{\tilde{a}}{\tilde{q}}$ always has a denominator $\tilde{q}$ satisfying $\tilde{q} \equiv \pm a^{-1}_q \mod q $. In fact, it is simple to observe that $\tilde{q} = \max \{ \bar{q} \in [1, \Lambda] \cap \mathbb{Z}: \bar{q} \equiv  \pm a_q^{-1} \mod q \}$. Moreover, because for fixed $q$ and variable $a \in \mathbb{Z}^\times_q$, the extension of $I(a,q)$ to the left or right varies inversely with $\tilde{q}$, for every $\Lambda \in 2^{\mathbb{N}}, \tau \in \mathbb{T}$, $1 \leq q \leq \Lambda$ there are $N_1(q, \Lambda, \tau), N_2(q, \Lambda, \tau) \in \mathbb{Z}^\times_q $ such that
\begin{align*}
 \{ a \in \mathbb{Z}^\times_q : I(a,q) \ni \tau\} = \{ a \in \mathbb{Z}^\times_q:  N_1( q, \Lambda, \tau) \leq a_q^{-1}  \leq N_2(q, \Lambda, \tau)\}. 
\end{align*}
\end{proof}
We shall also need the following bound for restricted Kloosterman sums, which follows from the exposition before the statement of Theorem 3 in \cite{Bourgain1}. 
\begin{lemma}\label{L:Kloosterman}
For every $\delta>0$ there is $A= A(\delta)$ such that for all 
\begin{align}\label{Est:Klooster}
\left| \sum_{\substack{ a \in \mathbb{Z}^\times_q \\ x \leq a^{-1}_q \leq y} } e^{2 \pi i a \frac{b}{q}} \right| \leq A (b, q)^{1/2} q^{1/2+\delta}  \qquad \forall b,q ,x,y \in \mathbb{N}:  x  \leq y < q
\end{align}
where $(b,q)$ is the gcd of $b$ and $q$.

\end{lemma}
The most important feature of the above estimate is that it is uniform in $x$ and $y$.

 \begin{proof}
Our plan is to show for all $d \geq 5$ and $(\frac{1}{p}, \frac{1}{r}) \in \mathcal{S}_*(d)$ there is $\eta =\eta(d,p,r)>0$ such that  
\begin{align}
\left \lVert  \sup_{\Lambda \leq \lambda < 2 \Lambda} |\mathscr{A}_\lambda - \mathscr{B}_\lambda| \right \rVert_{\ell^p \to \ell^{r^\prime}}  \leq & A \Lambda^{-\eta} \Lambda^{d( \frac{1}{r^\prime} - \frac{1}{p} )} \qquad  \forall \Lambda \in 2^{\mathbb{N}} \label{Est:R1} \\
\left \Vert  \sup_{\Lambda \leq \lambda < 2 \Lambda} |\mathscr{B}_\lambda - \mathscr{C}_\lambda | \right \rVert_{\ell^p \to \ell^{r^\prime}} \leq & A \Lambda^{-\eta} \Lambda^{d( \frac{1}{r^\prime} - \frac{1}{p})} \qquad \forall \Lambda \in 2^{\mathbb{N}}. \label{Est:R2}
\end{align}
While we do not need the additional $\Lambda^{-\eta}$ for the improving estimates, it is in the proof of the sparse bound that such estimates become useful.  To this end, we define for every $Q, \Lambda \in 2^{\mathbb{N}}: Q \leq \Lambda < \lambda < 2 \Lambda, \tau \in \mathbb{R}, \xi \in \mathbb{T}^d$
\begin{align}
\mu_{Q,\tau, \lambda} (\xi):=& \sum_{Q \leq q < 2Q}   \sum_{\substack{ a \in \mathbb{Z}^\times_q\\  a^{-1}_q \in [N_1(q,\Lambda, \tau),  N_2(q , \Lambda, \tau)]}}e^{- 2 \pi i \frac{a}{q} \lambda^2} \mu_{a/q, \tau, \lambda}(\xi) \label{def:muQ} \\ 
\mu_{a/q,\tau, \lambda}(\xi) :=&      \sum_{\ell \in \mathbb{Z}^d} G(a/q, \ell) (1- \Phi_q (\xi - \ell /q)) e^{- \pi |\xi - \ell/q|^2/2 ( \epsilon - i \tau)} \label{Def:mu} \\ 
\gamma_{Q,\tau, \lambda}(\xi):=&\sum_{Q \leq q< 2Q} \sum_{\substack{ a \in \mathbb{Z}^\times_q  \\ a^{-1}_q \not \in [N_1(q,\Lambda, \tau), N_2(q , \Lambda, \tau)]}}  e^{-2 \pi i \frac{a}{q} \lambda^2}\gamma_{a/q,\tau, \lambda}(\xi)  \label{Def:gammaQ}\\ 
\gamma_{a/q,\tau, \lambda}(\xi) :=&   \sum_{\ell \in \mathbb{Z}^d/q\mathbb{Z}^d} G(a/q, \ell)\Phi_q (\xi - \ell /q) e^{- \pi |\xi - \ell/q|^2/2 ( \epsilon - i \tau)}. \label{Def:gamma}
\end{align}
From \eqref{Def:a(lambda)} and \eqref{Def:b(lambda)}, it follows that for all $\Lambda \in 2^{\mathbb{N}}, x \in \mathbb{Z}^d,$ and $f : \mathbb{Z}^d \to \mathbb{C}$
\begin{align}
\sup_{\Lambda \leq \lambda < 2 \Lambda} |(\mathscr{A}_\lambda - \mathscr{B}_\lambda)f(x)| \leq& \frac{A}{\Lambda^{d-2}} \sum_{Q \leq \Lambda} \int_{|\tau| \leq  \frac{A}{Q \Lambda}} \frac{\sup_{\Lambda \leq \lambda < 2 \Lambda}| f * \check{\mu}_{Q, \tau, \lambda}(x) |}{(\Lambda^{-4} +  \tau^2)^{d/4}}d \tau \label{Est:A-B} \\ 
\sup_{\Lambda \leq \lambda < 2 \Lambda} |(\mathscr{B}_\lambda - \mathscr{C}_\lambda)f(x)| \leq& \frac{A}{\Lambda^{d-2}} \sum_{Q \leq \Lambda} \int_{|\tau| \geq \frac{A}{\Lambda Q}} \frac{\sup_{\Lambda \leq \lambda < 2 \Lambda}| f * \check{\gamma}_{Q, \tau, \lambda}(x) |}{(\Lambda^{-4} +  \tau^2)^{d/4}} d \tau.  \label{Est:B-C}
\end{align}
Letting $T_m$ denote the convolution operator with corresponding symbol $m \in L^\infty(\mathbb{T}^d)$, it suffices for the proof of \eqref{Est:R1} to show for every $\delta >0$ there is $A$ so that for all $\tau: |\tau| \leq \frac{A}{\Lambda Q}$
 \begin{align}
\left \lVert \sup_{\Lambda \leq \lambda < 2 \Lambda} | T_{\mu_{Q, \tau, \lambda}} | \right \rVert_{\ell^2 \to \ell^2} \leq & A Q^2 (\Lambda^{-4} + \tau^2)^{d/4}
\label{Est:Rl21} 
\end{align}
and for all $(\frac{1}{p}, \frac{1}{r}) \in \mathcal{T}_*(d)$
\begin{align}
\left \lVert \sup_{\Lambda \leq \lambda < 2 \Lambda} \left| T_{\mu_{Q, \tau, \lambda}} \right| \right \rVert_{\ell^p \to \ell^{r^\prime}} \leq & A Q^{3/2} \Lambda^{d+\delta} (\Lambda^{-4} + \tau^2)^{d/4}  \Lambda^{d(\frac{1}{r^\prime}- \frac{1}{p})} \label{Est:Rl22} 
\end{align}
where $\mathcal{T}_*(d)$ is defined to be the subset of $[0,1]^2$ for which there exists $(\frac{1}{p_1}, \frac{1}{r_1}) \in \mathcal{T}(d)$ and $(\frac{1}{p_2}, \frac{1}{q_2}) \in \{  \max\{ \frac{1}{p} , \frac{1}{q} \} = 1 \}$ such that $( \frac{1}{p}, \frac{1}{r}) = \frac{1}{2} \cdot (\frac{1}{p_1} + \frac{1}{p_2}, \frac{1}{q_1}+ \frac{1}{q_2} ).$

To see the sufficiency of \eqref{Est:Rl21} and \eqref{Est:Rl22} for showing \eqref{Est:R1}, we interpolate between \eqref{Est:Rl21} and \eqref{Est:Rl22} to find that for all $(\frac{1}{p}, \frac{1}{r}) \in \mathcal{S}_*(d)$ there is $A = (d,p,r)$ for which
\begin{align}\label{Est:Interp2.0}
\left \lVert \sup_{\Lambda \leq \lambda < 2 \Lambda} | T_{\mu_{Q, \tau, \lambda}}|  \right \rVert_{\ell^p \to \ell^{r^\prime}} \leq & A Q^{2 \frac{3}{d-1} } Q^{\frac{d-4}{d-1}} (\Lambda^{-4} + \tau^2)^{d/4} \Lambda^{(d+1/2) \frac{d-4}{d-1}}.
\end{align}
Using \eqref{Est:A-B} and Minkowski's inequality allows us to write down 
\begin{align}
& \left \Vert  \sup_{\Lambda \leq \lambda < 2 \Lambda} |\mathscr{A}_\lambda - \mathscr{B}_\lambda | \right \rVert_{\ell^p \to \ell^{r^\prime}}    \leq  \frac{A}{\Lambda^{d-2}} \sum_{Q \leq \Lambda} \int_{|\tau| \leq \frac{A}{\Lambda Q}} \frac{\left \lVert T_{\mu_{Q, \tau, \lambda}} \right \rVert_{\ell^p \to \ell^{r^\prime}} }{(\Lambda^{-4} +  \tau^2)^{d/4}}   d \tau \label{Est:InterpAux}
\end{align}
and then substituting \eqref{Est:Interp2.0} into \eqref{Est:InterpAux} yields $\forall \left(\frac{1}{p}, \frac{1}{r}\right) \in \mathcal{S}_*(d)$
\begin{align*}
\left \Vert  \sup_{\Lambda \leq \lambda < 2 \Lambda} |\mathscr{A}_\lambda - \mathscr{B}_\lambda | \right \rVert_{\ell^p \to \ell^{r^\prime}} \leq  A \Lambda^{- \frac{3d}{2(d-1)}} . 
\end{align*}
The above estimate is even better than \eqref{Est:R1}. 
To prove  \eqref{Est:R2}, it suffices to show that for every $\delta>0$ there is $A$ such that for all $\tau : |\tau| \geq \frac{A}{\Lambda Q}$ 
\begin{align}
\left \lVert \sup_{\Lambda \leq \lambda < 2 \Lambda} | T_{\gamma_{Q, \tau, \lambda}} |  \right \rVert_{\ell^2 \to \ell^2} \leq & A Q^{2-d/2} \label{Est:Rl31} 
\end{align}
and for all $(\frac{1}{p}, \frac{1}{r})  \in \mathcal{T}_*(d)$ and $\delta >0$
\begin{align}
\left \lVert \sup_{\Lambda \leq \lambda < 2 \Lambda} | T_{\gamma_{Q, \tau, \lambda}} | \right \rVert_{\ell^p \to \ell^{r^\prime}} \leq & A Q^{3/2} \Lambda^{d+ \delta}  (\Lambda^{-4} + \tau^2)^{d/4} \Lambda^{d(\frac{1}{r^\prime} - \frac{1}{p})}. \label{Est:Rl32} 
\end{align}
Indeed, interpolating between \eqref{Est:Rl31} and \eqref{Est:Rl32} yields that for all $( \frac{1}{p} , \frac{1}{r} )  \in \mathcal{S}_*(d)$
\begin{align}
 \left \lVert \sup_{\Lambda \leq \lambda < 2 \Lambda} | T_{\gamma_{Q, \tau, \lambda}} | \right \rVert_{\ell^p \to \ell^{r^\prime}}  \leq & AQ^{(2 - d/2)\frac{3}{d-1}} Q^{\frac{3}{2}\frac{d-4}{d-1}} \Lambda^{-\eta}  (\Lambda^{-4} + \tau^2)^{ \frac{d}{4} \frac{d-4}{d-1} }  \Lambda^{d \frac{d-4}{d-1}}. \label{Est:Interp20} 
\end{align}
Using \eqref{Est:B-C} and Minkowski's inequality allows us to write down 
\begin{align}
& \left \Vert  \sup_{\Lambda \leq \lambda < 2 \Lambda} |\mathscr{B}_\lambda - \mathscr{C}_\lambda |  \right \rVert_{\ell^p \to \ell^{r^\prime}} \nonumber  \\ \leq & \frac{A}{\Lambda^{d-2+\eta}} \sum_{Q \leq \Lambda} \int_{|\tau| \geq \frac{A}{\Lambda Q}} \frac{ \left \lVert \sup_{\Lambda \leq \lambda < 2 \Lambda} |T_{\gamma_{Q, \tau, \lambda}}| \right \rVert_{\ell^p \to \ell^{r^\prime}}}{(\Lambda^{-4} +  \tau^2)^{d/4}}   d \tau \label{Est:InterpAux1}
\end{align}
and then substituting \eqref{Est:Interp20} into \eqref{Est:InterpAux1} yields for all $\left(\frac{1}p, \frac{1}r \right) \in \mathcal{S}_*(d)$ there is some $\eta = \eta(d,p,r)>0$ such that
\[
\left \Vert  \sup_{\Lambda \leq \lambda < 2 \Lambda} |\mathscr{B}_\lambda - \mathscr{C}_\lambda | \right \rVert_{\ell^p \to \ell^{r^\prime}} \leq A \Lambda^{-\eta} \Lambda^{d(\frac{1}{r^\prime} - \frac{1}{p})}
\] 
as desired.

Our motivation for introducing $\mathcal{T}_*(d)$ is that it serves as the set of interpolation ``midpoints" between the the boundary set $\{ (p,r): \max\{ \frac{1}p, \frac{1}r\} = 1\}$ and the region of continuous improving estimates, namely $\mathcal{T}(d)$, which turns out to be necessary because of an issue arising from the Kloosterman sum bound \eqref{Est:Klooster}. We now proceed to prove the $\ell^2 \to \ell^2 $ estimates \eqref{Est:Rl21},  \eqref{Est:Rl31} before the $\ell^p \to \ell^{r^\prime}$ estimates \eqref{Est:Rl22}, and \eqref{Est:Rl32}. Moreover, we shall highlight the precise place in the argument where the Kloosterman sum issue arises. 

We begin the proof of \eqref{Est:Rl21} with the triangle inequality: 
\begin{align} 
& \left \lVert \sup_{\Lambda \leq \lambda < 2 \Lambda} | T_{\mu_{Q, \tau, \lambda}} |  \right \rVert_{\ell^2 \to \ell^2}     \leq  \sum_{Q \leq q < 2Q}   \sum_{\substack{ a \in \mathbb{Z}^\times_q\\  a^{-1}_q  \in [N_1(q,\Lambda, \tau), \leq N_2(q , \Lambda, \tau)]}}  \left \lVert \sup_{\Lambda \leq \lambda < 2 \Lambda} | T_{\mu_{a,q, \tau, \lambda}} |  \right \rVert_{\ell^2 \to \ell^2}.  \label{Est:TrianRl1}
\end{align}
To handle the supremum over $\lambda$, we majorize
\begin{align}
& \sup_{\Lambda \leq \lambda < 2 \Lambda} | f*   \check{\mu}_{a,q, \tau, \lambda} |  \nonumber\\ \leq & |  f * \check{\mu}_{a,q, \tau, \Lambda} |  + \left( \int_{\Lambda}^{2 \Lambda} \frac{d}{d\lambda} | f * \check{\mu}_{a,q, \tau, \lambda}| ^2 d \lambda \right)^{1/2} \nonumber  \\ \leq&  |  f * \check{\mu}_{a,q, \tau, \Lambda} |  + \left( \int_{\Lambda}^{2 \Lambda}  \left | \frac{d}{d\lambda}  f * \check{\mu}_{a,q, \tau, \lambda} \right| ^2 d \lambda \right)^{1/2} \cdot \left( \int_{\Lambda}^{2 \Lambda}  |   f * \check{\mu}_{a,q, \tau, \lambda}| ^2 d \lambda \right)^{1/2}   . \label{Est:Var} 
\end{align}
Using the definition of $\check{\mu}_{a,q,\tau, \lambda}$ given by \ref{Def:mu}  and the basic Gauss sum estimate \eqref{Est:GS}, we observe
\begin{align}
\left \lVert  \mu_{a,q,\tau, \lambda} \right \rVert_{L^\infty(\mathbb{T}^d)} +    \left \lVert \lambda \frac{d}{d\lambda}  \mu_{a,q, \tau, \lambda} \right \rVert_{L^\infty(\mathbb{T}^d)} \leq A (\Lambda^{-4} + \tau^2)^{d/4}. \label{Est:Mult}
\end{align}
Finally, note that \eqref{Est:Rl21} follows from \eqref{Est:TrianRl1}, \eqref{Est:Var}, and \eqref{Est:Mult}.   
The proof of \eqref{Est:Rl31} is just as short. Again, begin with
\begin{align} 
& \left \lVert \sup_{\Lambda \leq \lambda < 2 \Lambda} |T_{\gamma_{Q, \tau, \lambda}} |  \right \rVert_{\ell^2 \to \ell^2}    \leq  \sum_{Q \leq q < 2Q}   \sum_{\substack{ a \in \mathbb{Z}^\times_q\\  a^{-1}_q \not  \in [N_1(q,\Lambda, \tau), \leq N_2(q , \Lambda, \tau)]}}  \left \lVert \sup_{\Lambda \leq \lambda < 2 \Lambda} | T_{\gamma_{a,q, \tau, \lambda}} |  \right \rVert_{\ell^2 \to \ell^2}.  \label{Est:TrianRl2}
\end{align}
To handle the supremum over $\lambda$, we  employ the pointwise bound
\begin{align}
& \sup_{\Lambda \leq \lambda < 2 \Lambda} | f*   \check{\gamma}_{a,q, \tau, \lambda} |  \nonumber\\ \leq &  |  f * \check{\gamma}_{a,q, \tau, \Lambda} |  + \left( \int_{\Lambda}^{2 \Lambda}  \left | \frac{d}{d\lambda}  f * \check{\gamma}_{a,q, \tau, \lambda} \right| ^2 d \lambda \right)^{1/2} \cdot \left( \int_{\Lambda}^{2 \Lambda}  |   f * \check{\gamma}_{a,q, \tau, \lambda}| ^2 d \lambda \right)^{1/2}  . \label{Est:Var2} 
\end{align}
Using the definition of $\check{\gamma}_{a,q,\tau, \lambda}$ given by \eqref{Def:gamma} and the basis Gauss sum estimate  \eqref{Est:GS}, we observe 
\begin{align}
\left \lVert  \gamma_{a,q,\tau, \lambda} \right \rVert_{L^\infty(\mathbb{T}^d)}+   \left \lVert \lambda  \frac{d}{d\lambda}  \gamma_{a,q, \tau, \lambda} \right \rVert_{L^\infty(\mathbb{T}^d)} \leq A Q^{-d/2}. \label{Est:Mult2}
\end{align}
Like before, \eqref{Est:Rl31} follows from \eqref{Est:TrianRl2}, \eqref{Est:Var2}, and \eqref{Est:Mult2}.

It therefore remains to show \eqref{Est:Rl21} and \eqref{Est:Rl22}. As an intermediate goal, we prove for every $f : \mathbb{Z}^d \to \mathbb{C}, \delta>0,$ and $x \in \mathbb{Z}^d$
\begin{align}
|f*(1_{|\cdot| \not = \lambda} \check{\mu}_{Q, \tau, \lambda} )(x) | \leq & A  Q^{3/2}  \Lambda^{d+\delta} (\Lambda^{-4} + \tau^2)^{d/4}  \left(  |f|^{1+\delta} * \check{\Phi}_{\Lambda} (x) \right)^{\frac{1}{1+\delta}}  \label{Est:Point1} \\
|f*(1_{|\cdot| \not = \lambda} \check{\gamma}_{Q, \tau, \lambda } ) (x) | \leq &  A  Q^{3/2}  \Lambda^{d+\delta} (\Lambda^{-4} + \tau^2)^{d/4}  \left(  |f|^{1+\delta} * \check{\Phi}_{\Lambda} (x) \right)^{\frac{1}{ 1+\delta}}.  \label{Est:Point2}
\end{align}
To this end, we use \eqref{def:muQ} and \eqref{Def:gammaQ} to observe
\begin{align}
\frac{\check{\mu}_{Q, \tau, \lambda}(n)}{ ( \epsilon - i \tau)^{d/2}}  =& \sum_{Q \leq q < 2 Q }  \sum_{\substack{ a \in \mathbb{Z}^\times_q\\ a^{-1}_q \in [N_1(q,\Lambda, \tau) , N_2(q , \Lambda, \tau)]}}e^{2 \pi i a\frac{ (|n|^2 - \lambda^2)}{q}} e^{- \pi |\cdot|^2 (\epsilon - i \tau)} *(\delta_0 - \check{\Phi}_q)(n) \label{Est:Point3} \\ 
\frac{\check{\gamma}_{Q, \tau, \lambda}(n)}{ ( \epsilon - i \tau)^{d/2}}  =& \sum_{Q \leq q < 2 Q } \sum_{\substack{ a \in \mathbb{Z}^\times_q\\  a^{-1}_q \not \in [ N_1(q,\Lambda, \tau), N_2(q , \Lambda, \tau)]}}e^{2 \pi i a\frac{ (|n|^2 - \lambda^2)}{q}} e^{- \pi |\cdot|^2 (\epsilon - i \tau)} * \check{\Phi}_q(n), \label{Est:Point4}
\end{align}
where the convolutions appearing on the right sides of the above display are taken in the continuous sense and $\delta_0$ is the Dirac delta function. By Lemma \ref{L:Kloosterman} and the above kernel identities, we obtain for every $\delta >0$ and $n \in \mathbb{Z}^d:|n| \not = \lambda$
\begin{align*}
\frac{|\check{\mu}_{Q, \tau, \lambda}(n)|  +  |\check{\gamma}_{Q, \tau, \lambda}(n)| }{(\Lambda^{-4} + \tau^2)^{d/4}}\leq& A Q^{1/2+\delta} \Lambda^d \left[ \sum_{Q \leq q < 2 Q }  (|n|^2 - \lambda^2, q) ^{1/2} \right]   |\check{\Phi}_\Lambda (n)|. 
\end{align*}
The right side of above display is $O_\delta( Q^{3/2} \Lambda^{d+\delta} |\check{\Phi}_\Lambda(n)| )$ for every $\delta>0$, and so the estimates \eqref{Est:Point1} and \eqref{Est:Point2} hold.  As a corollary, we have for all $(\frac{1}p, \frac{1}{r}) :\frac{1}{p} + \frac{1}{r} \geq 1$ 
\begin{align}
\left \lVert \sup_{\Lambda \leq \lambda < 2 \Lambda} | T_{\mathcal{F}_{\mathbb{Z}^d}(1_{|\cdot| \not = \lambda} \check{\mu}_{Q, \tau, \lambda})}|   \right \rVert_{\ell^p \to \ell^{r^\prime}} \leq A Q^{3/2} \Lambda^{d+\delta} (\Lambda^{-4} + \tau^2)^{d/4}   \Lambda^{d (\frac{1}{r^\prime}- \frac{1}{p})}  \label{Est:5} 
\end{align}
A complication emerges in the case when $|n| = \lambda$, as there is no cancellation in the Kloosterman sum. For this reason, the argument uses the region of interpolation ``midpoints" given by $\mathcal{T}_*(d)$, and in particular, the fact that for all $(\frac{1}p, \frac{1}{r}) \in \mathcal{T}_*(d)$
\begin{align}
\left \lVert \sup_{\Lambda \leq \lambda < 2 \Lambda} | T_{\mathcal{F}_{\mathbb{Z}^d}(1_{|\cdot| \not = \lambda} \check{\mu}_{Q, \tau, \lambda})}|   \right \rVert_{\ell^p \to \ell^{r^\prime}} \leq A Q^{3/2} \Lambda^{d} (\Lambda^{-4} + \tau^2)^{d/4}   \Lambda^{d(\frac{1}{r^\prime} - \frac{1}{p})}.  \label{Est:6}
\end{align}
Indeed, this claim follows by observing from \eqref{Est:Point3} that for all $(\frac{1}p, \frac{1}{r}) : \max\{ \frac{1}p, \frac{1}r\}=1$ 
\begin{align}
\left \lVert \sup_{\Lambda \leq \lambda < 2 \Lambda} | T_{\mathcal{F}_{\mathbb{Z}^d}(1_{|\cdot| \not = \lambda} \check{\mu}_{Q, \tau, \lambda})}|   \right \rVert_{\ell^p \to \ell^{r^\prime}} \leq A Q^2 \Lambda^d(\Lambda^{-4} + \tau^2)^{d/4}   \Lambda^{d(\frac{1}{r^\prime} -\frac{1}{p})}\label{Est:10}
\end{align}
while a direct transference to the continuous case yields for all $(\frac{1}p, \frac{1}{r}) \in \mathcal{T}(d)$
\begin{align}
\left \lVert \sup_{\Lambda \leq \lambda < 2 \Lambda} | T_{\mathcal{F}_{\mathbb{Z}^d}(1_{|\cdot| \not = \lambda} \check{\mu}_{Q, \tau, \lambda})}|   \right \rVert_{\ell^p \to \ell^{r^\prime}} \leq A Q \Lambda^d (\Lambda^{-4} + \tau^2)^{d/4}  \Lambda^{d(\frac{1}{r^\prime} -\frac{1}{p})}, \label{Est:11} 
\end{align}
where we recall $\mathcal{T}(d)$ from the statement of Theorem \ref{Thm:0}. 
Indeed, using \eqref{Est:Point3} and \eqref{Est:Point3}, we obtain for every $n \in \mathbb{Z}^d: |n| = \lambda$ 
\begin{align*}
 \frac{|\check{\mu}_{Q, \tau, \lambda}(n)|  +  |\check{\gamma}_{Q, \tau, \lambda}(n)| }{(\Lambda^{-4} + \tau^2)^{d/4}}  \leq& A Q^2 \\ \leq& A Q^2 \Lambda^{d-1} \frac{1_{dist(\cdot, \{|\cdot|= \lambda\}) \leq 1}(n)}{\Lambda^{d-1}} \\   \leq& A Q \Lambda^d \frac{1_{dist(\cdot, \{|\cdot|= \lambda\}) \leq 1}(n)}{\Lambda^{d-1}}.  
\end{align*}
Interpolating between \eqref{Est:10} and \eqref{Est:11} yields \eqref{Est:6}. 
On account of \eqref{Est:5} and \eqref{Est:6}, the triangle inequality then yields \eqref{Est:Rl22}, that is for all $(\frac{1}p, \frac{1}{r}) \in \mathcal{T}_*(d)$
\begin{align}
 \left \lVert \sup_{\Lambda \leq \lambda < 2 \Lambda} |T_{\mu_{Q, \tau, \lambda}}|   \right \rVert_{\ell^p \to \ell^{r^\prime}} \leq A Q^{3/2} \Lambda^{d+\delta} (\Lambda^{-4} + \tau^2)^{d/4}   \Lambda^{d(\frac{1}{r^\prime} -\frac{1}{p})}. 
\end{align}
A similar line of reasoning yields \eqref{Est:Rl32},  that is for all $(\frac{1}p, \frac{1}{r}) \in \mathcal{T}_*(d)$ and $\delta >0$
\begin{align}
\left \lVert \sup_{\Lambda \leq \lambda < 2 \Lambda} | T_{\gamma_{Q, \tau, \lambda}}|   \right \rVert_{\ell^p \to \ell^{r^\prime}} \leq A Q^{3/2} \Lambda^{d+\delta} (\Lambda^{-4} + \tau^2)^{d/4}   \Lambda^{d(\frac{1}{r^\prime} - \frac{1}{p})}.  \label{Est:7}
\end{align}

 \end{proof}

  \begin{corollary}\label{Cor:1}
 
Let $d \geq 5$ and $(\frac{1}{p}, \frac{1}{r}) \in \mathcal{R}_*(d)$. Then there is $A=A(d,p,r)$ such that \begin{align*}
\left \lVert  \sup_{\Lambda \leq \lambda < 2 \Lambda} |\mathscr{A}_\lambda | \right \rVert_{\ell^p \to \ell^{r^\prime}}  \leq& A \Lambda^{d(1/r^\prime - 1/p)} \qquad \forall \Lambda \in 2^{\mathbb{N}}.
\end{align*}

  \end{corollary}
  \begin{proof}
  By Theorems \ref{Thm:MainImp} and \ref{Thm:ResImp}, it follows that for all $d \geq 5$ and $(\frac{1}{p}, \frac{1}{r}) \in \mathcal{
  S}_*(d)$ there is $A=A(d,p,r)$ such that 
  \begin{align}\label{Est:Diag}
\left| \left| \sup_{\Lambda \leq \lambda < 2 \Lambda} |\mathscr{A}_\lambda |    \right| \right|_{\ell^p \to \ell^{r^\prime}}  \leq& A \Lambda^{d(1/r^\prime - 1/p)} \qquad \forall \Lambda \in 2^{\mathbb{N}}. 
\end{align}
Then interpolate \eqref{Est:Diag} with the trivial $\ell^\infty \to \ell^\infty$ bound for $\sup_{\Lambda \leq \lambda < 2 \Lambda} |\mathscr{A}_\lambda|$.  
  \end{proof}

\section{New Sparse Bounds for $\sup_{\lambda} |\mathscr{A}_\lambda|$}
We begin by recalling Theorem 16 from \cite{Kesler5}. 
\begin{theorem}\label{Restr-Sparse}
Let $T$ be an operator on $\mathbb{Z}^d$ satisfying the property that for some $p,r: \frac{1}{p} + \frac{1}{r} >1$  there is an $A$ such that for all finite sets $E_1, E_2 \subset \mathbb{Z}^d$ and $|f| \leq 1_{E_1}, |g| \leq 1_{E_2}$, there is a sparse collection $\mathcal{S}$ such that 
\begin{align*}
 \left| \langle T f, g\rangle \right| \leq A \Lambda_{\mathcal{S}, p,r} (1_{E_1}, 1_{E_2}).
\end{align*}
Then for every $\tilde{p}> p, \tilde{r}>r$ such that $\frac{1}{\tilde{p}} + \frac{1}{\tilde{r}} >1$ there is $A$ such that for all finitely supported $f,g : \mathbb{Z}^d \to \mathbb{C}$ there is a sparse collection $\mathcal{S}$ such that
\begin{align*}  \left| \langle T f , g \rangle \right| \leq A \Lambda_{\mathcal{S}, \tilde{p},\tilde{r}} (f, g).
\end{align*}
\end{theorem}
There is a continuous version of Theorem 16, which is shown via a similar proof.
\begin{theorem}\label{Thm:Sparse}
For all $d \geq 5$

\begin{align}
\left \lVert \sup_{\lambda \in \tilde{\Lambda}} | \mathscr{A}_\lambda|: (p,r)  \right \rVert < \infty \qquad  \forall \left(\frac{1}p, \frac{1}r\right) \in \mathcal{R}_*(d). \label{Est:Sp10}
\end{align}
\end{theorem}
\begin{proof}
By Theorem \ref{Restr-Sparse}, it suffices to show 
\begin{align}
\left \lVert \sup_{\lambda \in \tilde{\Lambda}} | \mathscr{A}_\lambda|: (p,r)  \right \rVert_{restricted}< \infty  \qquad \forall \left(\frac{1}p, \frac{1}r\right) \in \mathcal{R}_*(d). \label{Est:Sp101}
\end{align}
To this end, let $(\frac{1}p, \frac{1}{r}) \in \mathcal{R}_*(d)$ along with $|f| \leq 1_{E_1}$ and $|g| \leq 1_{E_2}$. Pick a dyadic cube $E \subset \mathbb{Z}^d$ such that $3E \supset E_1 \cup E_2$ and let $\mathcal{Q}(E)$ be those dyadic stopping cubes of $3E$ such that 
\begin{flalign*}
\left \langle 1_{3E}  \sup_{\lambda \in \tilde{\Lambda}}  \mathcal{A}_\lambda  1_{E_1}  \right \rangle_{Q,p} \geq &  C_0 \langle 1_{E_1} \rangle_{3E,p} \\ 
\langle 1_{E_1} \rangle_{5Q,1}  \geq & C_0 \langle 1_{E_1} \rangle_{3E, 1}\\ 
\langle 1_{E_2} \rangle_{5Q,1} \geq & C_0 \langle 1_{E_2} \rangle_{3E,1}  \\ 
\left \langle \left( \sum_{N \in 2^{\mathbb{N}}} |M_p (P_N \chi_{E_1})|^2 \right)^{1/2} \right \rangle_{3Q, p(1+\delta)} \geq & C_0 \langle \chi_{E_1} \rangle_{3E,p(1+\delta)}\\ 
\left \langle \left( \sum_{\Lambda \in 2^{\mathbb{N}}} |M_r ( \chi_{S_\Lambda \cap E_2})|^2 \right)^{1/2} \right \rangle_{3Q, r(1+\delta)} \geq & C_0 \langle \chi_{E_2} \rangle_{3E,r(1+\delta)}
\end{flalign*}  
for some $\delta = \delta(d,p,r)>0$ to be determined later. 
It follows that for large enough constant $C_0$ we have that the collection $\mathcal{Q}(E)$ satisfies $\sum_{Q \in \mathcal{Q}(E)} |Q| \leq \frac{|E|}{100}$. Moreover, note that by the stopping condition, 
\begin{align}
1_{ (\bigcup_{\mathcal{Q}(E)} Q)^c}  1_{3E} \sup_{\lambda \in \tilde{\Lambda}}  \mathcal{A}_\lambda  1_{E_1} \leq A \langle 1_{E_1} \rangle_{3E,p}. 
\end{align}
Therefore, we observe
\begin{align*}
\left| \langle \sup_\lambda | \mathscr{A}_\lambda f | , g \rangle \leq \right| & \langle \sup_\lambda \mathcal{A}_\lambda 1_{E_1}  , 1_{E_2} \rangle \\ \leq & \sum_{Q \in \mathcal{Q}(E)} \langle 1_{Q}  \sup_\lambda \mathscr{A}_\lambda 1_{E_1} , 1_{E_2} \rangle + A \langle 1_{E_1} \rangle_{3E,p} \langle 1_{E_2} \rangle_{3E,r} |E|. 
\end{align*}
In fact, it shall suffice to show
\begin{align}\label{Est:Goal20}
 \left \langle \sum_{Q \in \mathcal{Q}(E)}1_{Q}  \sup_\lambda \mathscr{A}_\lambda( 1_{(3Q)^c}1_{E_1}) , 1_{E_2} \right \rangle  \leq A \langle 1_{E_1} \rangle_{3E,p} \langle 1_{E_2} \rangle_{3E,r} |E|. 
\end{align}
Indeed, we shall be able to recurse on
\begin{align*}
\sum_{Q \in \mathcal{Q}(E)} \langle 1_{Q}  \sup_\lambda \mathscr{A}_\lambda (1_{3Q}1_{E_1} ), 1_{E_2} \rangle
\end{align*}
by letting each $Q$ play the role that $E$ played before and thereby generate a pre-sparse collection $\tilde{\mathcal{S}}= \{3Q\}$ for which 
\begin{align*}
| \langle \sup_\lambda | \mathcal{A}_\lambda f | , g \rangle | \leq A \sum_{3Q \in \tilde{S}} \langle 1_{E_1} \rangle_{3Q,p} \langle 1_{E_2} \rangle_{3Q,r} |Q|.
\end{align*}
As there is a sparse form $\Lambda_{S, p,r} (1_{E_1}, 1_{E_2})$ which dominates $\Lambda_{\tilde{S}, p,r} (1_{E_1}, 1_{E_2})$, \eqref{Est:Sp101} holds. By restricted interpolation, we note that is actually suffices to prove \eqref{Est:Goal20} for $(\frac{1}{p}, \frac{1}{r})$ arbitrarily close to the $4$ extremal points in $\mathcal{R}_*(d)$. That is, we just need that for each of the extremal points there is a sequence of $(\frac{1}p, \frac{1}r)$ tuples converging to the extremal point for which \eqref{Est:Goal20} always holds. For the extremal point $\mathcal{R}_{d,1,*}= (0,1)$, \eqref{Est:Goal20} follows immediately from trivial estimate $\left \lVert \sup_{\lambda \in \tilde{\Lambda}} |\mathscr{A}_\lambda|  \right \rVert _{\ell^\infty \to \ell^\infty} <\infty.$ To handle $\mathcal{R}_{d,2,*} = (\frac{d-2}{d}, \frac{2}{d})$, we take any sequence  $\{1/p_j\}_{j \in \mathbb{N}}$ contained in the interval $(0, \frac{d-2}{d})$ for which $1/p_j \to_- \frac{d-2}{d}$. Then  $(1/p_j, 1-1/p_j) \to (\frac{d-2}{d}, \frac{2}{d})$ and \eqref{Est:Goal20} holds at each $(1/p_j, 1-1/p_j)$ since for all $j \in \mathbb{N}$
\begin{align*}
 \left\lVert \sum_{Q \in \mathcal{Q}(E)}1_{Q}  \sup_\lambda \mathscr{A}_\lambda( 1_{(3Q)^c}1_{E_1}) \right \rVert_{\ell^{p_j}(\mathbb{Z}^d)} \leq & \left\lVert \sum_{Q \in \mathcal{Q}(E)}1_{Q}  \sup_\lambda \mathscr{A}_\lambda( 1_{E_1}) \right \rVert _{\ell^{p_j}(\mathbb{Z}^d)} \\ +& \left\lVert \sum_{Q \in \mathcal{Q}(E)}1_{Q}  \sup_\lambda \mathscr{A}_\lambda( 1_{3Q}1_{E_1}) \right \rVert_{\ell^{p_j}(\mathbb{Z}^d)}  \\ \leq& A |E_1|^{1/p_j}. 
 \end{align*}
We now proceed to the two remaining extremal points $\mathcal{R}_{d,3,*} = (\frac{d-2}{d}, \frac{d-2}{d})$ and 
 \begin{flalign*}
 & \mathcal{R}_{d,4,*} = \left(   \frac{1}{2} \left(\frac{ d^2 - d}{d^2 + 1} + 1\right )\frac{d-4}{d-1} + \frac{3/2}{d-1}  , \frac{1}{2}\left (\frac{ d^2 - d+2}{d^2 + 1} + 1\right)\frac{d-4}{d-1} + \frac{3/2}{d-1}  \right).&
 \end{flalign*}
 To this end, we assign
\[
\mathcal{Q}_\Lambda(E) :=\left. \{ R \in \mathcal{D} \right|  \Lambda \leq \ell(R) \leq 2 \Lambda ~and~  \exists Q \in \mathcal{Q}(E): Q \subset R \}
\]
and observe
\begin{flalign}\label{Est:Goal21.5}
&  \left \langle \sum_{Q \in \mathcal{Q}(E)}1_{Q}  \sup_{\Lambda} \sup_{\Lambda \leq \lambda < 2 \Lambda} \left| \mathscr{A}_\lambda( 1_{(3Q)^c}1_{E_1})\right| , 1_{E_2} \right \rangle \\ \leq &  \left \langle  \sum_{\substack{ \Lambda \in 2^{\mathbb{N}} \\  \Lambda \leq A \ell(E) }} \sum_{R \in \mathcal{Q}_\Lambda(E)} 1_{S_\Lambda \cap R} \sup_{\Lambda \leq \lambda < 2 \Lambda} \left| \mathscr{A}_\lambda(1_{5R}1_{E_1})\right| , 1_{E_2} \right \rangle \nonumber &  
 \end{flalign}
 where $S_\Lambda:= \{ x \in \mathbb{Z}^d: \arg \max_\Lambda \sup_{\Lambda \leq \lambda < 2\Lambda} \mathscr{A}_\lambda 1_{E_1}(x)  = \Lambda \}$, so that $\{S_{\Lambda}\}_{\Lambda \in 2^{\mathbb{N}}}$ forms a partition of $\mathbb{Z}^d$. 
 Our first claim is that for all $(\frac{1}{p}, \frac{1}{r}) \in \mathcal{Q}_*(d)$
\begin{flalign}\label{Est:Goal22}
  &\left \langle \sum_{\substack{ \Lambda \in 2^{\mathbb{N}} \\  \Lambda \leq A \ell(E) }}\sum_{R \in \mathcal{Q}_\Lambda(E)} 1_{S_\Lambda \cap R} \sup_{\Lambda \leq \lambda < 2 \Lambda} \left| \mathscr{C}_\lambda(1_{5R} 1_{E_1})\right| , 1_{E_2} \right \rangle \leq  A \langle 1_{E_1} \rangle_{3E,p} \langle 1_{E_2} \rangle_{3E,r} |E|. &
  \end{flalign}
Indeed, we begin by observing 
\begin{flalign*}
&\left \lVert \sum_{\substack{ \Lambda \in 2^{\mathbb{N}} \\  \Lambda \leq A \ell(E) }}\sum_{R \in \mathcal{R}_\Lambda(E)} 1_{S_\Lambda \cap R} \sup_{\Lambda \leq \lambda < 2 \Lambda} \left| \mathscr{C}_{Q,\lambda}(1_{5R} 1_{E_1})\right|  \right \rVert_{\ell^2 \to \ell^2}  \leq A Q^{2-d/2}&
\end{flalign*}
and break apart the above display by showing $O( \langle 1_{E_1} \rangle_{3E,p} \langle 1_{E_2} \rangle_{3E,r} |E|)$ bounds for 
\begin{flalign*}
& I:= \left \langle \sum_{\substack{ \Lambda \in 2^{\mathbb{N}} \\  \Lambda \leq A \ell(E) }}\sum_{R \in \mathcal{Q}_\Lambda(E)} 1_{S_\Lambda \cap R} \sup_{\Lambda \leq \lambda < 2 \Lambda} \left| \mathscr{C}_\lambda(1_{(5R)^c} 1_{E_1})\right| , 1_{E_2} \right \rangle&
\end{flalign*}
 and 
\begin{flalign*}
&II :=  \left \langle \sum_{\substack{ \Lambda \in 2^{\mathbb{N}} \\  \Lambda \leq A \ell(E) }}\sum_{R \in \mathcal{Q}_\Lambda(E)} 1_{S_\Lambda \cap R} \sup_{\Lambda \leq \lambda < 2 \Lambda} \left| \mathscr{C}_\lambda( 1_{E_1})\right| , 1_{E_2} \right \rangle.&
\end{flalign*}
For $I$, we obtain for all $1 \leq q \leq \Lambda$, $a \in \mathbb{Z}^\times_q$, and $R \in \mathcal{R}_\Lambda(E)$
\begin{align*}
 1_R \cdot \sup_{\Lambda \leq \lambda < 2 \Lambda} |\mathscr{C}_{a,q,\lambda} (1_{(5R)^c}  1_{E_1} )  |  \leq   \frac{A}{\Lambda^{5}}  1_R \cdot 1_{E_1}* |\check{\Phi}_\Lambda | .  
 \end{align*}
Now summing on $1 \leq q \leq \Lambda$ and $a \in \mathbb{Z}^\times_q$ yields
\begin{align*}
\sum_{\substack{ \Lambda \in 2^{\mathbb{N}} \\  \Lambda \leq A \ell(E) }}\sum_{R \in \mathcal{Q}_\Lambda(E)} 1_{S_\Lambda \cap R} \sup_{\Lambda \leq \lambda < 2 \Lambda} \left| \mathscr{C}_\lambda(1_{(5R)^c} 1_{E_1})\right| \leq \frac{A}{\Lambda^3}  1_{R}\cdot ( 1_{E_1}* |\check{\Phi}_\Lambda |).
\end{align*}
But if $R \in \mathcal{Q}_\Lambda(E)$, $1_R \cdot ( 1_{E_1}* |\check{\Phi}_\Lambda |) \leq A \langle 1_{E_1} \rangle_{3E,1}$. 
Therefore,
\begin{align*}
I \leq A  \sum_{\substack{ \Lambda \in 2^{\mathbb{N}} \\  \Lambda \leq A \ell(E) }}\frac{1}{\Lambda^3} \left \langle 1_{E_1} * |\check{\Phi}_\Lambda | , 1_{E_2}  \right \rangle \leq A \langle 1_{E_1} \rangle_{3E, 1} \langle 1_{E_2} \rangle_{3E,1} |E|,
\end{align*}
which is even stronger bound than those near $\mathcal{S}_{d,3,*}$ and $\mathcal{S}_{d,4,*}$. Therefore, to conclude the proof of \eqref{Est:Goal22}, it suffices to obtain $II \leq A \langle 1_{E_1} \rangle_{3E,p} \langle 1_{E_2} \rangle_{3E,r} |E|$.
To this end,  we may observe from \eqref{Est:C2} that
\begin{flalign*}
 & \left \langle \sum_{\substack{ \Lambda \in 2^{\mathbb{N}} \\  \Lambda \leq A \ell(E) }}\sum_{R \in \mathcal{Q}_\Lambda(E)} 1_{S_\Lambda \cap R}  \sup_{\Lambda \leq \lambda < 2 \Lambda} \left| \mathscr{C}_{\lambda} 1_{E_1} \right|, 1_{E_2}   \right \rangle \leq A Q^{2-d/2} \langle 1_{E_1}  \rangle_{3E,2} \langle 1_{E_2} \rangle_{3E,2} |E|.&
\end{flalign*}
Hence, for \eqref{Est:Goal22} we only need to prove for all $(\frac{1}{p}, \frac{1}{r}) \in \mathcal{T}(d)$ and $\delta>0$
\begin{flalign*}
&  \left \langle \sum_{\substack{ \Lambda \in 2^{\mathbb{N}} \\  \Lambda \leq A \ell(E) }}\sum_{R \in \mathcal{Q}_\Lambda(E)} 1_{S_\Lambda \cap R} \sup_{\Lambda \leq \lambda < 2 \Lambda} \left| \mathscr{C}_{Q,\lambda} 1_{E_1} \right|, 1_{E_2}   \right \rangle  \leq A Q^{1+\delta}  \langle 1_{E_1}  \rangle_{3E,p} \langle 1_{E_2} \rangle_{3E,r} |E|. &
\end{flalign*}
To this end, we use estimate \eqref{Est:1} to majorize the left side of the above display by
\begin{flalign*}
&A   \sum_{l \geq 0} \frac{Q^{1+\delta}}{2^{dl}} \left \langle\sum_{\substack{ \Lambda \in 2^{\mathbb{N}} \\  \Lambda \leq A \ell(E) }}\sum_{R \in \mathcal{Q}_\Lambda(E)} 1_{S_\Lambda \cap R}  \sup_{\Lambda \leq \lambda < 2 \Lambda} \left|  \left(1_{E_1}  * d \sigma _\lambda *(|\check{\Phi}_1| + |\check{\Phi}_{2^l Q}| ) \right)^{\frac{1}{1+\delta}}\right|, 1_{E_2}   \right \rangle.&
\end{flalign*}
We now transfer to the continuous setting. To this end, we assign for every $F \subset \mathbb{Z}^d$ the function $\chi_{F}: \mathbb{R}^d \rightarrow \{0,1\}$ as follows:\begin{align*}
\chi_{F} (x) =1 \qquad for~x \in \bigcup_{n \in F} \prod_{j=1}^d [n_j - 1/2 , n_j +1/2] \\ 
\chi_{F} (x) = 0 \qquad for ~x \not \in \bigcup_{n \in F} \prod_{j=1}^d [n_j - 1/2 , n_j +1/2] \\ 
\end{align*}
and let $\tilde{F} := supp (\chi_F)$. 
It follows that
\begin{flalign*}
&  \left \langle \sum_{\substack{ \Lambda \in 2^{\mathbb{N}} \\  \Lambda \leq A \ell(E) }}  \sum_{l \geq 0}   \sum_{R \in \mathcal{Q}_\Lambda(E)}  \frac{1_{S_\Lambda(E) \cap R}}{2^{dl}} \sup_{\Lambda \leq \lambda < 2 \Lambda} \left|  \left(1_{E_1}  * d \sigma _\lambda *(|\check{\Phi}_1| + |\check{\Phi}_{2^l Q}| ) \right)^{\frac{1}{ 1+\delta}}\right|, 1_{E_2}   \right \rangle_{\mathbb{Z}^d}&  \nonumber \\ \leq & A  \left \langle \sum_{\substack{ \Lambda \in 2^{\mathbb{N}} \\  \Lambda \leq A \ell(E) }} \sum_{l \geq 0}  \sum_{R \in \mathcal{Q}_\Lambda(E)}   \frac{\chi_{S_\Lambda \cap R} }{2^{dl}} \sup_{\Lambda \leq \lambda < 2 \Lambda} \left|  \left(\chi_{E_1}  * d \sigma _\lambda *(|\check{\Phi}_1| + |\check{\Phi}_{2^l Q}| ) \right)^{\frac{1}{1+\delta}}\right|, \chi_{E_2}   \right \rangle_{\mathbb{R}^d}.& 
\end{flalign*}
The contribution to the lower level of  the above display when $l: 2^l > \frac{\Lambda}{Q}$ is easily handled via the observation that for every $R \in \mathcal{Q}_\Lambda(E)$
\begin{align*}
\chi_{S_\Lambda \cap R}  \sup_{\Lambda \leq \lambda < 2 \Lambda}   \chi_{E_1}  * d \sigma _\lambda * |\check{\Phi}_{2^l Q}| \leq  A \chi_{S_\Lambda \cap R}  \cdot(   \chi_{E_1}  * |\check{\Phi}_{2^lQ }| ) \leq A \langle \chi_{E_1} \rangle_{3E,1}, 
\end{align*}
so that the upper bound in \eqref{Est:Goal21} improves to $A \langle \chi_{E_1} \rangle_{3E, 1+\epsilon} \langle \chi_{E_2} \rangle_{3E,1} |E|$, which is even better than required. Hence, it suffices to obtain
\begin{flalign}\label{Est:Goal21}
& \left \langle \sum_{\substack{ \Lambda \in 2^{\mathbb{N}} \\  \Lambda \leq A \ell(E) \\\Lambda \geq 2^l Q}}\sum_{R \in \mathcal{Q}_\Lambda(E)}  \chi_{S_\Lambda  \cap R}   \sup_{\Lambda \leq \lambda < 2 \Lambda} \left|  \left(\chi_{E_1}  * d \sigma _\lambda *(|\check{\Phi}_1| + |\check{\Phi}_{2^l Q}| ) \right)^{\frac{1}{1+\delta}}\right|, \chi_{E_2}   \right \rangle_{\mathbb{R}^d} \\  \leq &  A \langle\chi_{E_1} \rangle_{3E,p} \langle \chi_{E_2} \rangle_{3E,r} |E| \nonumber&
\end{flalign}
uniformly in $l \geq 0$ and for all $(\frac{1}{p}, \frac{1}{r}) \in \mathcal{T}(d)$.  To show \eqref{Est:Goal21},  we introduce a standard Littlewood-Paley decomposition  $\{P_{2^k}\}_{k \in \mathbb{Z}}$ where the frequency support of each $P_{2^k}$ is inside $[-2^{k+1} , -2^{k-1} ] \bigcup [2^{k-1}, 2^{k+1}]$. Set $P_{< 2^l} = \sum_{k < l} P_{2^k}$. Then we majorize the upper line of \eqref{Est:Goal21} as
\begin{flalign*}
& \sum_{N \in 2^{\mathbb{N}}}  \left \langle \sum_{\substack{ \Lambda \in 2^{\mathbb{N}} \\  \Lambda \leq A \ell(E) \\ \Lambda \geq 2^l Q }}\sum_{R \in \mathcal{Q}_\Lambda(E)}  \chi_{S_\Lambda(E) \cap R}   \sup_{\Lambda \leq \lambda < 2 \Lambda}  \left| P_{N/ \Lambda} (\chi_{E_1}) * d \sigma _\lambda *(|\check{\Phi}_1| + |\check{\Phi}_{2^l Q}| ) \right|^{\frac{1}{ 1+\delta}}, \chi_{E_2}   \right \rangle_{\mathbb{R}^d} \\+&   \left \langle \sum_{\substack{ \Lambda \in 2^{\mathbb{N}} \\  \Lambda \leq A \ell(E) \\\Lambda \geq 2^l Q}}\sum_{R \in \mathcal{Q}_\Lambda(E)}  \chi_{S_\Lambda(E) \cap R}   \sup_{\Lambda \leq \lambda < 2 \Lambda} \left| P_{< \frac{1}{\Lambda} } (\chi_{E_1} )* d \sigma _\lambda *(|\check{\Phi}_1| + |\check{\Phi}_{2^l Q}| ) \right|^{\frac{1}{1+\delta}}, \chi_{E_2}   \right \rangle_{\mathbb{R}^d}& \\ =: &\left[ \sum_{N \in 2^{\mathbb{N}}} I_N  \right]+ II.  &
\end{flalign*}
For $II$, it is a simple matter to bound for $\Lambda \geq 2^l Q$ and $R \in \mathcal{Q}_\Lambda(E)$
\begin{align*}
1_R \left| P_{< \frac{1}{\Lambda} } (\chi_{E_1} )* d \sigma _\lambda *(|\check{\Phi}_1| + |\check{\Phi}_{2^l Q}| ) \right| \leq A 1_R \cdot\chi_{E_1} * |\check{\Phi}_{\Lambda}|  \leq A \langle \chi_{E_1} \rangle_{3E,1}
\end{align*}
in which case $II$ is dominated by $A \langle \chi_{E_1} \rangle_{3E, 1+\epsilon} \langle \chi_{E_2} \rangle_{3E,1} |E|$. To show \eqref{Est:Goal21}, it now suffices to prove that for all $(\frac{1}p, \frac{1}r) \in \mathcal{T}(d)$ and $N \in 2^{\mathbb{N}}$ there is $\eta = \eta(d, p,r)>0$ for which
 \begin{align*}
  I_N \leq A N^{-\eta} \langle \chi_{E_1} \rangle_{3E,p} \langle \chi_{E_2} \rangle_{3E,r} |E|.
  \end{align*}
For this, we first observe for all $R \in \mathcal{Q}_\Lambda(E)$ and $2^l Q \leq \Lambda \leq A \ell(E)$
\begin{flalign*}
& 1_{R} \cdot (\chi_{(5R)^c} (P_{N/ \Lambda} (\chi_{E_1}) ) * d \sigma _\lambda *(|\check{\Phi}_1| + |\check{\Phi}_{2^l Q}| ) \leq A\cdot | P_{N/\Lambda} \chi_{E_1}|* |\check{\Phi}_{\Lambda}|  \leq A \langle \chi_{E_1} \rangle_{3E,1}&
\end{flalign*}
so that 
\begin{flalign*}
& \left \langle \sum_{\substack{ \Lambda \in 2^{\mathbb{N}} \\  \Lambda \leq A \ell(E) \\ \Lambda \geq 2^l Q }}\sum_{R \in \mathcal{Q}_\Lambda(E)}  \chi_{S_\Lambda(E) \cap R}   \sup_{\Lambda \leq \lambda < 2 \Lambda}   \left|  (\chi_{(5R)^c} (P_{N/ \Lambda} (\chi_{E_1}) ) * d \sigma _\lambda *(|\check{\Phi}_1| + |\check{\Phi}_{2^l Q}| ) \right|^{\frac{1}{1+\delta}}, 1_{E_2} \right \rangle &\\ &\leq A  \langle \chi_{E_1} \rangle_{3E,1+\delta} \langle \chi_{E_2} \rangle_{3E,1} |E|. &
 \end{flalign*}
 Next, having handled the off-diagonal piece, we finish with the diagonal contribution. To this end, we observe from the decay of $\widehat{d \sigma}_\lambda$ that for all $\Lambda$ and $R \in \mathcal{Q}_\Lambda(E)$ 
\begin{flalign*}
& \left \langle \sup_{\Lambda \leq \lambda < 2 \Lambda}   \left|  (\chi_{5R} (P_{N/ \Lambda} (\chi_{E_1}) ) * d \sigma _\lambda *(|\check{\Phi}_1| + |\check{\Phi}_{2^l Q}| ) \right|^{\frac{1}{1+\delta}} \right \rangle_{R,2(1+\delta)} &\\ &\leq A N^{\frac{1-d/2}{1+\sigma(\delta)}}  \langle P_{N/\Lambda} \chi_{E_1} \rangle_{3R, 2(1+\delta)} &
\end{flalign*}
where $\sigma(\delta) \to 0 $ as $\delta \to 0$. 
Moreover, by Theorem \ref{Thm:0}, it follows that for any $(\frac{1}{p}, \frac{1}{r}) \in \mathcal{T}(d),R \in \mathcal{Q}_\Lambda (E)$ and $\delta >0$
\begin{flalign*}
 & \left  \langle    \sup_{\Lambda \leq \lambda < 2 \Lambda}   \left| (\chi_{5R} P_{N/ \Lambda} (\chi_{E_1}) )* d \sigma _\lambda *(|\check{\Phi}_1| + |\check{\Phi}_{2^l Q}| ) \right|^{1/(1+\delta)} \right \rangle_{5R, p (1+\delta)} \\ \leq & A \langle P_{N/\Lambda} \chi_{E_1} \rangle_{R, r(1+\delta)}. &
 \end{flalign*}
 Therefore, for all $(\frac{1}p, \frac{1}r) \in \mathcal{T}(d)$ there is $\eta = \eta(d,p,r) >0$ so that
 \begin{flalign*}
 &  \left \langle \sup_{\Lambda \leq \lambda < 2 \Lambda}   \left|  (\chi_{5R} (P_{N/ \Lambda} (\chi_{E_1}) ) * d \sigma _\lambda *(|\check{\Phi}_1| + |\check{\Phi}_{2^l Q}| ) \right|^{1/(1+\delta)} , \chi_{ S_\Lambda(E) \cap E_2}\right  \rangle_{\mathbb{R}^d} &  \\ \leq & A N^{-\eta}  \langle P_{N/\Lambda} \chi_{E_1} \rangle_{5R, p} \langle \chi_{S_\Lambda(E) \cap E_2} \rangle_{R,r} |R|. &
 \end{flalign*}
 To show estimate \eqref{Est:Goal21}, it suffices to obtain for all $(\frac{1}p, \frac{1}r) \in \mathcal{T}(d)$ and $\delta >0$
 \begin{flalign*}
   &\sum_{\substack{ \Lambda \in 2^{\mathbb{N}} \\  \Lambda \leq A \ell(E)}}\sum_{R \in \mathcal{Q}_\Lambda(E)}  \langle P_{N/\Lambda} \chi_{E_1} \rangle_{5 R, p} \langle \chi_{S_\Lambda(E) \cap E_2} \rangle_{ R,r} |R|\leq A \langle \chi_{E_1} \rangle_{3E,p(1+\delta)} \langle \chi_{E_2} \rangle_{3E,r(1+\delta)} |E|.& 
 \end{flalign*}
But, letting $\mathbb{R} (E) := \bigcup_{\Lambda} \bigcup_{\mathcal{Q}_\Lambda(E)} \{R\}$, we note that $|\mathbb{R}(E)| \leq A |E|$ and the left side of the above display is majorized for every $\delta >0$ by
\begin{flalign*}
&\int_{\mathbb{R}(E)} \sum_{\Lambda \leq N}  M^{\mathcal{Q}(E)} _p (P_{N/\Lambda }\chi_{E_1} ) (x)  M^{\mathcal{Q}(E)}_r (\chi_{S_\Lambda \cap E_2} ) (x)   dx& \\  \leq& \int_{\mathbb{R}(E)}\left(  \sum_{\Lambda \leq N}  \left| M^{\mathcal{Q}(E)}_p (P_{N/\Lambda }\chi_{E_1} ) (x) \right|^2 \right)^{1/2}  \left( \sum_{\Lambda \geq N} | M^{\mathcal{Q}(E)}_r (\chi_{S_\Lambda \cap E_2} ) (x)|^2 \right)^{1/2}   dx   & \\  \leq& A \langle \chi_{E_1} \rangle_{3E,p(1+\delta)} \langle \chi_{E_2} \rangle_{3E, r(1+\delta)}  |E|
\end{flalign*}
for all $\delta >0$. 
To control the contribution of the residual term, we note that on account of \eqref{Est:R1},\eqref{Est:R2}, and the stopping conditions, that for all $(\frac{1}{p}, \frac{1}{r}) \in \mathcal{S}_*(d)$ there is $\eta = \eta(d,p,r)>0$ for which
\begin{flalign}\label{Est:Goal24}
&  \left \langle  \sum_{\substack{ \Lambda \in 2^{\mathbb{N}} \\  \Lambda \leq A \ell(E) }} \sum_{R \in \mathcal{R}_\Lambda(E)}1_{S_\Lambda \cap R}   \sup_{\Lambda \leq \lambda < 2 \Lambda} \left| \mathscr{R}_\lambda(1_{5R}1_{E_1})\right| , 1_{E_2} \right \rangle  \\   \leq &\sum_{\substack{ \Lambda \in 2^{\mathbb{N}} \\  \Lambda \leq A \ell(E) }} \left \langle \sum_{R \in \mathcal{R}_\Lambda(E)}1_{R}  \sup_{\Lambda \leq \lambda < 2 \Lambda} \left| \mathscr{R}_\lambda( 1_{5R}1_{E_1})\right| , 1_{E_2} \right \rangle \nonumber  & \\ \leq &  A\sum_{\substack{ \Lambda \in 2^{\mathbb{N}} \\  \Lambda \leq A \ell(E) }} \Lambda^{-\eta}  \langle 1_{E_1} \rangle_{3E,p} \langle 1_{E_2} \rangle_{3E,r} |E| \nonumber \\ \leq& A \langle 1_{E_1} \rangle_{3E,p} \langle 1_{E_2} \rangle_{3E,r} |E|. \nonumber&
\end{flalign}
The proof of Theorem \ref{Thm:Sparse} follows from combining \eqref{Est:Goal21.5}, \eqref{Est:Goal22}, and \eqref{Est:Goal24}.

\end{proof}

 \bibliography{newspherical1}	
\bibliographystyle{amsplain}	
 \end{document}